\documentclass[preprint,11pt,numbers,sort&compress]{elsarticle}

\usepackage{soul}
\usepackage{psfrag}
\usepackage{a4wide}
\usepackage{rotating}
\usepackage{color}
\usepackage{amssymb}
\usepackage{amsmath}
\usepackage{amsthm}
\usepackage{verbatim}
\usepackage{cancel}
\usepackage{nicematrix}
\usepackage{multirow}

\newcommand{\f}[1]{\mathbf{#1}}
\newcommand{\ab}[1]{\boldsymbol{#1}}
\def\bfm#1{\boldsymbol{#1}}
\newcommand{\bb}[1]{\bfm{#1}}

\newcommand{\R}{\mathbb R}

\newcommand{\V}{\mathcal{V}}

\newcommand{\W}{\mathcal{W}}

\newcommand{\LL}{i_0}
\newcommand{\RR}{i_1}
\newcommand{\g}{f}
\newcommand{\gC}{f}
\newcommand{\sm}{s}
\newcommand{\sS}{s}
\newcommand{\ot}{\theta}
\newcommand{\Side}{\tau}
\newcommand{\dd}{\partial}

\DeclareMathOperator{\Span}{span}

\newtheorem{thm}{Theorem}

\theoremstyle{definition}
\newtheorem{ex}{Example}

\newtheorem{rem}{Remark}

\newproof{pf}{proof}
\advance\textheight by 0.4cm
\advance\topmargin by -0.2cm

\definecolor{gold}{rgb}{1,0.7,0}
\definecolor{dred}{rgb}{0.92,0,0}
\definecolor{dgreen}{rgb}{0,0.6,0}

\def\new{\color{blue}}

\begin{document}

\begin{frontmatter}

\title{A $C^\sS$-smooth mixed degree and regularity isogeometric spline space over planar multi-patch domains}

\cortext[cor]{Corresponding author}

\author[vil]{Mario Kapl}
\ead{m.kapl@fh-kaernten.at}

\author[slo1]{Alja\v z Kosma\v c}
\ead{aljaz.kosmac@iam.upr.si}

\author[slo1]{Vito Vitrih\corref{cor}}
\ead{vito.vitrih@upr.si}

\address[vil]{Department of Engineering $\&$ IT, Carinthia University of Applied Sciences, Villach, Austria}

\address[slo1]{IAM and FAMNIT, University of Primorska, Koper, Slovenia}

\begin{abstract}
We construct over a given bilinear multi-patch domain a novel $C^s$-smooth mixed degree and regularity isogeometric spline space, which possesses the degree $p=2s+1$ and regularity $r=s$ in a small neighborhood around the edges and vertices, and the degree~$\widetilde{p} \leq p$ with regularity $\widetilde{r} = \widetilde{p}-1 \geq r$ in all other parts of the domain. Our proposed approach 
relies on the technique~\cite{KaVi20b}, which requires for the $C^s$-smooth isogeometric spline space a degree 
at least $p=2s+1$ on the entire multi-patch domain. Similar to \cite{KaVi20b}, the $C^s$-smooth mixed degree and regularity spline space is generated as the span of 
basis functions that correspond to the 
individual
patches, edges and vertices of the 
domain. The 
reduction of degrees of freedom for the functions in the interior of the patches is achieved by introducing an appropriate mixed degree and regularity underlying spline space over
~$[0,1]^2$ to define the functions on the single patches. 
We further extend our construction 
with a few examples to the class of bilinear-like $G^s$ multi-patch parameterizations~\cite{KaVi17c, KaVi20b}, which enables the 
design of multi-patch domains 
having curved 
boundaries and
interfaces. Finally, the great potential of the $C^s$-smooth mixed degree and regularity isogeometric spline space for {performing} isogeometric analysis is demonstrated by several numerical examples of solving two particular high order partial differential equations, namely the biharmonic and triharmonic equation, via the isogeometric Galerkin method. 
\end{abstract}

\end{frontmatter}

\section{Introduction}

The basic idea of isogeometric analysis 
is to numercially solve 
a partial differential equation (PDE) by 
using the same spline 
space to describe the geometry of the physical domain and to represent the solution of the PDE, cf.~\cite{CottrellBook, HuCoBa04}. A common technique for describing complex domains, originated from computer aided geometric design, is the use of so-called multi-patch geometries with possibly extraordinary vertices, cf.~\cite{HoLa93, Fa97}. These are domains which are obtained by gluing several quadrilateral patches together and can possess vertices where less or more than four patches meet. To solve high-order PDEs over these multi-patch geometries, smooth spline functions are needed to represent the solution of the PDEs. Eg., in case of the Galerkin method, $C^1$-smooth 
isogeometric 
functions are required for solving 
$4$-th order PDEs, 
like the biharmonic equation, see e.g.~\cite{CoSaTa16, FaJuKaTa22, KaSaTa19b, KaBuBeJu16, NgKaPe15, WeTa22}, or the Kirchhoff-Love shell problem, see e.g.~\cite{FaVeKiKa23, Guo2015881, Pe15-2, kiendl-bazilevs-hsu-wuechner-bletzinger-10
}, and $C^2$-smooth 
functions for solving 
{$6$-th} order PDEs, 
{like} the triharmonic equation, see e.g.~\cite{KaVi17a, KaVi17b, TaDe14}. 

Most of the existing work for the 
{construction} of $C^{\sS}$-smooth ($\sS \geq 1$) isogeometric spline spaces over planar multi-patch 
{domains}, as studied in this paper, deals with the case~$\sS=1$, and can be 
{described by}
two possible strategies depending on whether the constructed functions are exactly $C^1$-smooth or just approximately $C^1$-smooth. In case of the approximate $C^1$-smoothness, the functions are obtained on the one hand indirectly by modifying the weak form of the problem such as for penalty methods, see e.g. \cite{HeJoPrWuKiHs2019, Leonetti2020
}, for Nitsche's methods, see e.g. \cite{
Guo2015881, Hu2018, Nguyen2014
} and for mortar methods, see e.g. \cite{Bouclier2017, Dornisch2015, Horger2019
}, or on the other hand directly by generating approximately $C^1$-smooth basis functions, see e.g.~\cite{SaJu21, TaTo22, WeTa22, WeTa21}. The concept of approximate smoothness for isogeometric spline functions over multi-patch geometries has been extended to 
{general} smoothness~$\sS \geq 1$ for mortar methods in~\cite{DiSchWoHe19, DiSchWoHe20}.

In case of the exact $C^1$-smoothness, the existing techniques can be divided into three possible approaches depending on the used type of parameterization for the underlying multi-patch geometry, namely first on a $C^1$-smooth multi-patch parameterization everywhere with singularities at the extraordinary vertices, see e.g.~\cite{NgPe16, ToSpHu17, WeLiQiHuZhCa22} and the related subdivision based approaches~\cite{Peters2, RiAuFe16, ZhSaCi18}, or with singularities at so-called scaling centers, see e.g.~\cite{ArReKlSi23, ReArSiKl23}, second on a $C^1$-smooth multi-patch parameterization with a $G^1$-smooth cap 
{around} 
extraordinary vertices, see e.g.~\cite{Pe15-2, KaPe17, KaPe18, NgKaPe15, WeFaLiWeCa23}, or third on a $G^1$-smooth multi-patch parameterization, which means 
{for a} planar multi-patch domain, just on a $C^0$-smooth multi-patch parameterization, see e.g.~\cite{BeMa14,BlMoVi17, BlMoXu20, ChAnRa18, ChAnRa19, KaBuBeJu16,KaSaTa21, mourrain2015geometrically}. The construction of exactly smooth isogeometric multi-patch spline spaces for a smoothness $\sS > 1$ has been studied for the first approach in \cite{ToSpHiHu16} for $s=2$ and for the third approach in \cite{KaVi17a,KaVi17b,KaVi17c,KaVi19a} for $s=2$ and in \cite{KaVi20b} for 
{general} $\sS \geq 1$. More details about generating $C^s$-smooth isogeometric spline functions over planar multi-patch geometries, in particular with the focus on the case $s=1$, can be found e.g. in the two survey articles~\cite{HuSaTaTo21, KaSaTa19b} as well as in the recent work~\cite{VeWeMaTaTo24}, where several techniques have been compared.

The third approach above for the 
{construction} of exactly $C^s$-smooth 
isogeometric multi-patch spline spaces is restricted for 
{a general} smoothness~$s$ to the technique~\cite{KaVi20b}, and employs {planar 
multi-patch domains with bilinear patch parameterizations} and their extension to the class of so-called {bilinear-like $G^{\sS}$ 
parameterizations}, cf.~\cite{KaVi17c, KaVi20b}. The class of bilinear-like $G^s$ multi-patch geometries is characterized by possessing the same linear gluing functions as bilinear multi-patch domains, coincides for the case $s=1$ with the class of analysis-suitable $G^1$ multi-patch geometries~\cite{CoSaTa16}, and 
{enables the construction} of multi-patch domains with {curved boundaries and patch interfaces}. A great benefit of the $C^s$-smooth isogeometric multi-patch spline space~\cite{KaVi20b} is that the spline space possesses optimal approximation properties as numerically shown 
in~\cite{KaVi20b} for $1 \leq \sS \leq 4$ by performing $L^2$ approximation, and in~\cite{KaKoVi24} for $s=4$ by solving the biharmonic equation via isogeometric collocation. However, a major drawback of the $C^{\sS}$-smooth spline space is that a quite high spline degree $p$ is needed, namely $p \geq 2\sS+1$. 

In this paper, we will present the construction of a novel $C^s$-smooth ($s \geq 1$) isogeometric spline space over planar bilinearly parameterized multi-patch geometries, and will also extend it {to 
bilinear-like $G^s$ 
parameterizations} on the basis of examples. Thereby, our proposed method will rely on the technique~\cite{KaVi20b}, and the generated $C^s$-smooth spline space will keep {the 
degree $p=2s+1$} and regularity $r=\sm$ just in a small neighborhood around the edges and vertices of the multi-patch domain, where it will be necessary, but will possibly reduce the degree and/or increase the regularity in all other parts of the domain, i.e. in the interior of each patch. 
This will be achieved by introducing an appropriate mixed degree and regularity underlying spline space over the unit square $[0,1]^2$ to define the functions on the single patches. The resulting $C^s$-smooth mixed degree and regularity isogeometric spline space will then allow to solve high-order PDEs with 
a much lower number of degrees of freedom compared to employing the spline space~\cite{KaVi20b} for the same level of refinement.

The idea of constructing exactly $C^s$-smooth isogeometric spline functions with a higher spline degree in the vicinity of the edges and/or vertices of a multi-patch geometry has been already studied for the second approach e.g. in \cite{Pe15-2,NgKaPe15,KaPe17,KaPe18}, and for the third approach, e.g in \cite{ChAnRa18, ChAnRa19}. However, all these methods are restricted to a smoothness $s=1$, and in case of the third approach the techniques~\cite{ChAnRa18, ChAnRa19} generate basis which can have a large support over one or more entire edges of the multi-patch domain. In contrast, our construction will work for an arbitrary smoothness $s\geq 1$, and all generated basis functions will possess a small local support. 

The potential of the constructed $C^s$-smooth mixed degree and regularity isogeometric spline space for solving high-order PDEs over planar multi-patch parameterizations will be demonstrated on the basis of several examples. 
In detail, we will use the isogeometric Galerkin method to solve the biharmonic equation with $C^1$-smooth functions, and the triharmonic equation with $C^2$-smooth functions. Thereby, we will consider two particular variants of the $C^s$-smooth mixed degree and regularity isogeometric spline space, namely on the one hand the $C^s$-smooth mixed degree spline space which possesses the same regularity $r=s$ everywhere and reduces the degree $p=2s+1$ in the interior of each patch to the smallest possible one, that is to $p=s+1$, and on the other hand the $C^s$-smooth mixed regularity spline space which has the same degree $p=2s+1$ everywhere and increases the regularity $r=s$ in the interior of each patch to the greatest possible one, that is to $r=2s$. While the first spline space, i.e. the $C^s$-smooth mixed degree spline space, will imply the lowest possible number of degrees of freedom, and will indicate optimal convergence rates with respect to the lower spline degree~$p=s+1$, the second spline space, i.e. the $C^s$-smooth mixed regularity spline space, will have a higher number of degrees of freedom but will indicate optimal convergence orders with respect to the greater spline degree $p=2s+1$. That means e.g. for the $C^s$-smooth mixed regularity spline space, that we will obtain with respect to $h$-refinement convergence rates of order $\mathcal{O}(h^{2s+2})$, $\mathcal{O}(h^{2s+1})$, $\mathcal{O}(h^{2s})$ for the $L^2$, $H^1$ and $H^2$-(semi)norms, respectively, and additionally $\mathcal{O}(h^{2s-1})$ for the $H^3$-seminorm in the case of the triharmonic equation.

The remainder of the paper is organized as follows. In Section~\ref{subsec:underlyingSpace}, we will present the mixed degree and regularity spline space on the unit square $[0,1]^2$, which will serve as the underlying spline space to define the $C^s$-smooth mixed degree and regularity isogeometric spline space over planar bilinearly parameterized multi-patch geometries, whose construction will be described in Section~\ref{sec:problem_statement}. Before the description of the $C^s$-smooth mixed degree and regularity isogeometric multi-patch spline space, we will also explain in Section~\ref{sec:problem_statement} the {setting of the} 
{considered} multi-patch 
{geometries} as well as the 
{basic principle}
of $C^s$-smooth isogeometric spline functions of mixed degree and regularity over {planar bilinear multi-patch 
domains}. 
Sections~\ref{sec:Galerkin} and~\ref{section_numerical_examples_Galerkin} 
will then be used to demonstrate the potential of the $C^s$-smooth mixed degree and regularity isogeometric spline space for solving high-order PDEs over 
{bilinear 
multi-patch geometries} including an extension {to 
bilinear-like $G^s$ multi-patch parameterizations}. More precisely, the biharmonic and triharmonic equations will be solved by 
{employing} the isogeometric Galerkin method 
{using} $C^1$ and $C^2$-smooth functions, respectively. Finally, 
{Section~\ref{sec:Conclusion} summarizes the main results of this study, and identifies issues that deserve further investigation}.

\section{A mixed degree and regularity spline space on the unit square $[0,1]^2$}  \label{subsec:underlyingSpace}

We will introduce a mixed degree and regularity spline space {on
~$[0,1]^2$} that will be used in Section~\ref{sec:problem_statement} as underlying spline space to construct a $C^\sS$-smooth ($\sS \geq 1$) isogeometric multi-patch spline space, and will show 
{in addition} some of the properties of the introduced mixed degree and regularity spline space. 


\subsection{Construction of the mixed degree and regularity spline space}

An underlying spline space to define a $C^{\sS}$-smooth 
isogeometric spline space usually has the same degree and regularity everywhere. While for the case of a one-patch domain, one can use an underlying spline space of bidegree $(p,p)$ with $p \geq \sS+1$ and regularity $r=\sS$, for the multi-patch case 
the degree~$p$ has to be at least $p=2\sS+1$, cf.~\cite{KaVi20b}. The goal of this section is 
to construct a particular mixed degree and regularity underlying spline space, which allows to define and to generate a mixed degree and regularity isogeometric multi-patch spline space which possesses the high degree~$p_2 = 2\sS +1$ with the low regularity $r_2=\sS$ just 
in a small neighborhood around the edges and vertices of the multi-patch domain, and can have a lower degree $p_1$ and a higher regularity $r_1$ with $\sS+1 \leq r_1 +1 \leq p_1 \leq 2\sS +1$ in all other parts of the domain. For this purpose, the introduced mixed degree and regularity underlying spline space {on
~$[0,1]^2$} will contain all those low degree $p_1$ and high regularity $r_1$ spline functions that have vanishing derivatives of order~$\ot \leq \sS$ at the boundary of $[0,1]^2$, and those high degree $p_2$ and low regularity $r_2$ spline functions whose support is contained only in the vicinity of the edges and vertices of $[0,1]^2$. Particular constructions of univariate mixed degree spline spaces and 
tensor-product extensions to the bivariate case have been already considered in \cite{ToSpHiHu16,ToSpHiMaHu2020,BeCaMo2017}. However, our construction will not be a simple tensor-product extension of a univariate mixed degree/regularity spline space but will remind more on a hierarchical construction as e.g.~in \cite{GiJuSp2012}.

We will denote by $\mathcal{S}_h^{p,r}([0,1])$ the spline space on the unit interval $[0,1]$ based on the knot vector \begin{equation*}  
(\underbrace{0,\ldots,0}_{(p+1)-\mbox{\scriptsize times}},
\underbrace{\textstyle \frac{1}{k+1},\ldots,\frac{1}{k+1}}_{(p-r) - \mbox{\scriptsize times}},\ldots, 
\underbrace{\textstyle \frac{k}{k+1},\ldots ,\frac{k}{k+1}}_{(p-r) - \mbox{\scriptsize times}},
\underbrace{1,\ldots,1}_{(p+1)-\mbox{\scriptsize times}}),
\end{equation*}
where $p$ and $r$ represent the degree and regularity of the space, respectively, while $h=\frac{1}{k+1}$ stands for the mesh size.
{Moreover,} let $\mathcal{S}_h^{\ab{p},\ab{r}}([0,1]^2)$ be the tensor-product spline space $\mathcal{S}_h^{p,r}([0,1])$ $\otimes 
\mathcal{S}_h^{p,r}([0,1])$ on the unit-square~$[0,1]^2$.
Additionally, we will refer to 
the B-splines of the spline spaces~$\mathcal{S}_h^{p,r}([0,1])$ and $\mathcal{S}_h^{\ab{p},\ab{r}}([0,1]^2)$
as $N_{j}^{p,r}$ and $N_{\ab{j}}^{\ab{p},\ab{r}} = N_{j_1,j_2}^{\ab{p},\ab{r}}=N_{j_1}^{p,r}N_{j_2}^{p,r}$, where $j,j_1,j_2=0,1,\ldots,n_{p,r}-1$, 
and $n_{p,r}= \dim \mathcal{S}_h^{p,r}([0,1]) = p+1+k(p-r)$,
respectively.


Although a more general definition of the mixed degree and regularity spline space would be possible, we will restrict ourselves 
to the practically most interesting case, namely to select degree $p_2$ and regularity $r_2$ as the lowest possible spline degree and regularity to fulfill $C^\sS$-smoothness conditions for the 
multi-patch case, i.e.
$
p_2 = 2\sS +1\; 
{\rm and} \; 
r_2 = \sS.
$
In this way, we will construct a particular underlying mixed degree and regularity spline space which will contain functions of degree $\ab{p}_2$ and regularity $\ab{r}_2$ in the vicinity of the vertices and edges of $[0,1]^2$, and functions of degree $\ab{p}_1$ and regularity $\ab{r}_1$, with 
$
\sS+1 \leq r_1 +1 \leq p_1 \leq 2\sS +1,
$
in the interior of $[0,1]^2$. 
Additionally, in order to reduce the number of degrees of freedom that are related to the functions in the interior of $[0,1]^2$, we will assume that $r_1=p_1-1$.  
Our mixed degree and regularity spline spaces will therefore satisfy throughout the paper the following assumptions:
\begin{equation}  \label{eq:boundsOnR1andP1}
p_2 = 2\sS +1, \quad r_2 = \sS, \quad r_1=p_1-1 \quad {\rm and} \quad  \sS+1 \leq p_1 \leq 2\sS +1.
\end{equation}

Let us first recall some relations between the univariate spline spaces $\mathcal{S}_h^{{p_1},{r_1}}([0,1])$ and $\mathcal{S}_h^{{p_2},{r_2}}([0,1])$. Since by \eqref{eq:boundsOnR1andP1} we have $p_1\leq p_2$ and $r_1 \geq r_2$, it follows that 
$$
\mathcal{S}_h^{{p_1},{r_1}}([0,1]) \subseteq \mathcal{S}_h^{{p_2},{r_2}}([0,1]),
$$ and we can represent any 
B-spline $N_i^{p_1,r_1}$ of the spline space $\mathcal{S}_h^{{p_1},{r_1}}([0,1])$ with respect to the 
B-splines $N_j^{p_2,r_2}$ of the spline space $\mathcal{S}_h^{{p_2},{r_2}}([0,1])$ as 
\begin{equation}  \label{eq:Bsplines_P1ToP2}
N_i^{p_1, r_1} = \sum_{j=0}^{n_{2}-1} \mu_j^{(i)} N_j^{p_2,r_2}, \quad i=0,1,\ldots, n_{1}-1, \; \mu_j^{(i)} \in \R,\, \mu_j^{(i)} \geq 0,
\end{equation}
where $n_i := n_{p_i,r_i}$.
By the partition of unity property for the B-splines 
of the spline spaces $\mathcal{S}_h^{p_i,r_i}([0,1])$, $i=1,2$, and by the linear independency of them, we have
$$
0 = 
\sum_{i=0}^{n_{1}-1} N_i^{p_1,r_1} - \sum_{j=0}^{n_{2}-1} N_j^{p_2,r_2} = 
\sum_{j=0}^{n_{2}-1} \left( \sum_{i=0}^{n_{1}-1} \mu_j^{(i)} -1 \right) N_j^{p_2,r_2},
$$
which implies
\begin{equation}  \label{eq:partitionMu}
   \sum_{i=0}^{n_{1}-1} \mu_j^{(i)} = 1, \quad j=0,\ldots,n_{2}-1. 
\end{equation} 
In order to generalize \eqref{eq:Bsplines_P1ToP2} and \eqref{eq:partitionMu} to the bivariate case let us define the 
two index sets
$$
\mathcal{J}_{i} = \{ \ab{j}= (j_1,j_2), \;\, j_1,j_2 = 0,1,\ldots, n_{i} -1 \}, \quad i=1,2.
$$
Then, we have
\begin{equation*}  \label{eq:Bsplines_P1ToP2_Bivariate}
N_{\ab{i}}^{\ab{p}_1,\ab{r}_1} = \sum_{\ab{j} \in \mathcal{J}_2} \mu_{\ab{j}}^{(\ab{i})}  N_{\ab{j}}^{\ab{p}_2,\ab{r}_2}, \quad \ab{i} \in \mathcal{J}_1, \quad \mu_{\ab{j}}^{(\ab{i})} := \mu_{j_1}^{(i_1)} \mu_{j_2}^{(i_2)} \in \R,\; \mu_{\ab{j}}^{(\ab{i})}>0,
\end{equation*}
and
\begin{equation}  \label{eq:partitionMu_Bivariate}
   \sum_{\ab{i} \in \mathcal{J}_1} \mu_{\ab{j}}^{(\ab{i})}  = 1, \quad \ab{j} \in \mathcal{J}_2.  
\end{equation}

%

We first 
introduce the mixed degree and regularity  spline space for the univariate case by explaining the construction of its basis. 
Recall \eqref{eq:boundsOnR1andP1}. 
We start with all 
B-splines $N_{i}^{{p}_1,{r}_1}$ from the space $\mathcal{S}_h^{{p}_1,{r}_1}([0,1])$, remove all of them that have a non-vanishing derivative of order $\ot \leq \sm$ at $0$ or $1$, and replace them with the 
B-splines $N_{i}^{{p}_2,{r}_2}$ from the space $\mathcal{S}_h^{{p}_2,{r}_2}([0,1])$ that have a non-vanishing derivative of order $\ot \leq \sm$ at $0$ or $1$.
In order to ensure the completeness of the resulting space we add some 
of the previously eliminated 
B-splines from the space $\mathcal{S}_h^{{p}_1,{r}_1}([0,1])$
after a truncation with respect to the space $\mathcal{S}_h^{{p}_2,{r}_2}([0,1])$ to guarantee that all derivatives of order $\ot \leq \sm$ 
vanish at $0$ and $1$.
Thus, the mixed degree and regularity spline space denoted by $\mathcal{S}_h^{({p}_1, {p}_2),{(r_1,r_2)}}([0,1]) 
$ is given by 
\begin{align*}  \label{eq:MixedSpace1D}
 \mathcal{S}_h^{({p}_1, {p}_2),{(r_1,r_2)}}([0,1]) & = \mathcal{S}_{1} ([0,1])  \oplus \mathcal{\overline{S}}_1 ([0,1]) \oplus \mathcal{S}_2 ([0,1])\\
 & = \Span \left\{ N_{i}^{{p}_1,{r_1}}, \; i = \sm+1,\ldots, n_{1}-\sm-2\right\}  \\
 & \oplus \Span \left\{ \overline{{N}}_{i}^{\,{p}_1,{r_1}}, \; i=1,\ldots,\sm,\, n_{1}-\sm-1, \ldots,n_{1}-2 \right\}\\
 &  \oplus \Span \left\{ N_{i}^{{p}_2,{r_2}}, \; i=0,\ldots,\sm,\, n_{2}-\sm-1, \ldots,n_{2}-1 \right\}, 
\end{align*}
where 
\begin{equation} \label{eq:truncation}
\overline{N}_{i}^{\,p_1,r_1} = 
\sum_{j=\sm+1}^{n_{2}-\sm-2} \mu_j^{(i)} N_j^{p_2,r_2}, \quad i=0,1,\ldots, n_{1}-1, \quad  \mu_j^{(i)} \geq 0,
\end{equation}
is a truncation of $N_i^{p_1,r_1}$ 
with respect to the space $\mathcal{S}_h^{{p}_2,{r_2}}([0,1])$ such that all derivatives of order $\ot \leq \sm$ 
vanish at $0$ and $1$, cf. Fig.~\ref{fig:SpacesMixed1D}.
Note that 
\begin{equation} \label{eq:truncationPrecise}
   \overline{N}_{i}^{\,p_1,r_1} = {N}_{i}^{p_1,r_1},  \;\; i=\sm+1,\ldots,n_{1}-\sm-2, 
   \quad {\rm and} \quad \overline{N}_{i}^{\,p_1,r_1} = 0, \;\; i = 0,n_{1}-1.
\end{equation}
\begin{figure}[h!]
        \centering
    \includegraphics[scale=0.33]{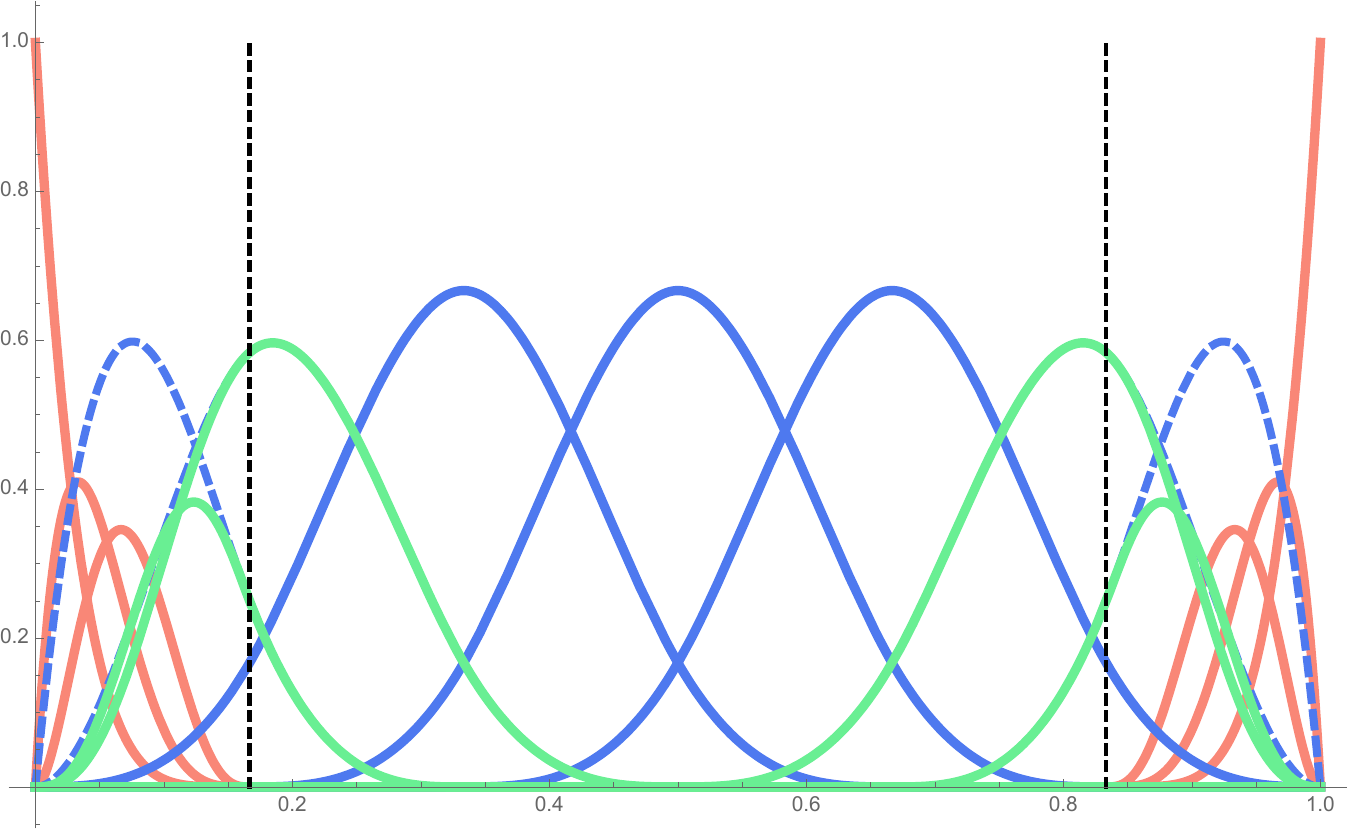}
    \hskip0.3em
    \includegraphics[scale=0.41]{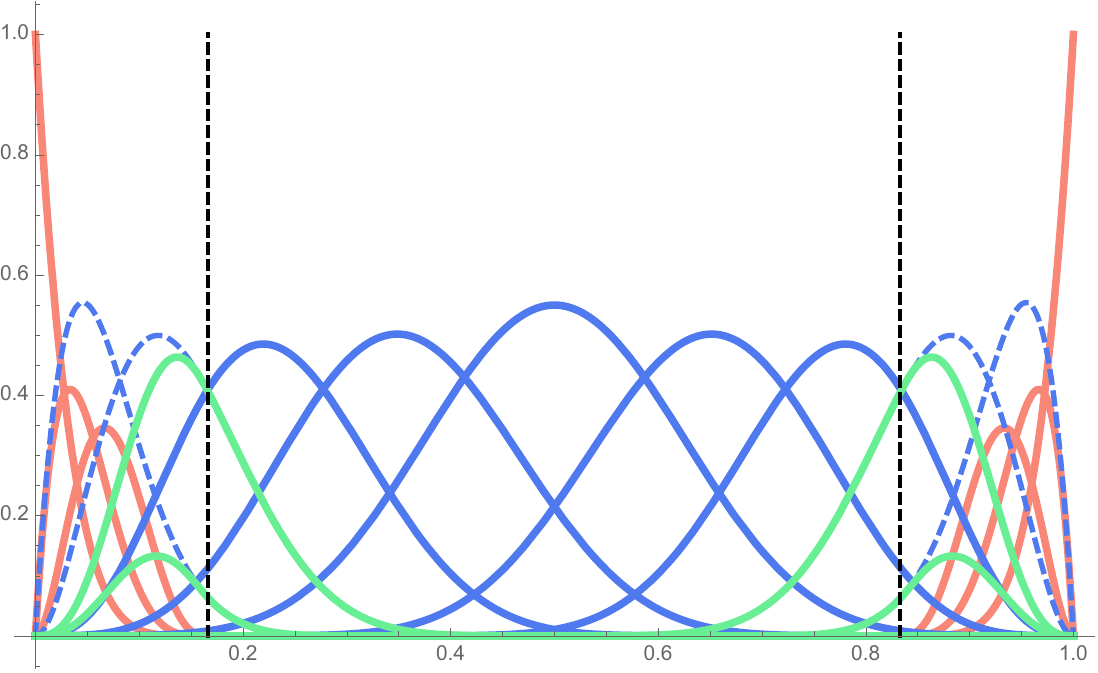}
    \caption{The basis functions from 
    the space $\mathcal{S}_h^{(p_1, p_2),{(r_1,r_2)}}([0,1])$, where $p_1=\sm+1$, $r_1=\sm$ (left), and $p_1=2\sm+1$, $r_1=2\sm$ (right), with $\sm=2$ and $k=5$. The blue and red functions are 
    B-splines from the spaces $\mathcal{S}_h^{{p}_1,{r_1}}([0,1])$ and $\mathcal{S}_h^{{p}_2,{r_2}}([0,1])$, respectively, while the green functions are the truncated ones. The dashed blue functions are the ones from $\mathcal{S}_h^{{p}_1,{r_1}}([0,1])$ before the truncation. The dashed gray vertical lines split the interval $[0,1]$ into $[0,h]$, $[h,1-h]$ and $[1-h,1]$. 
    }
    \label{fig:SpacesMixed1D}
\end{figure}
Since $n_1=p_1+1+k$, the subspace $\mathcal{S}_1([0,1])$ is nonempty if and only if
$$
k \geq 2(s+1)-p_1 \geq 1.
$$
Note that $k=1$ is only possible in the case $p_1=2\sm+1$ (which implies $r_1=2\sm$). 
The only B-spline of the space $\mathcal{S}_1([0,1])$ is in this case the B-spline $N_{\sm+1}^{2\sm+1,2\sm}$. However, this function is a linear combination of the B-splines from the space $\overline{\mathcal{S}}_1([0,1])$. More precisely, the following relation holds
$$
  N_{\sm+1}^{2\sm+1,2\sm} = \sum_{i=0}^{\sm-1} (-1)^i \left(  \overline{N}_{\sm-i}^{2\sm+1,2\sm} + \overline{N}_{\sm+2+i}^{2\sm+1,2\sm} \right) \in \overline{\mathcal{S}}_1([0,1]).
$$
Since we would like to avoid such situations, we will assume from now on that 
\begin{equation}  \label{eq:assumptionOnK}
  k \geq \max \{ 2(\sm+1)-p_1,2\},  \quad
  {\rm which}\; {\rm implies}\quad h\leq \min \left\{  \frac{1}{2(\sm+1)-p_1+1},\frac{1}{3}\right\}.
\end{equation}

Let us extend now the idea of the mixed degree and regularity spline space to the bivariate case, again by constructing its basis. 
Note that this will not be just a standard tensor-product extension. We start with the 
B-splines $N_{j_1,j_2}^{\ab{p}_1,\ab{r}_1}$ from the spline space $\mathcal{S}_h^{\ab{p}_1,\ab{r}_1}([0,1]^2)$ and skip all of them that have a non-vanishing derivative of order $\ot \leq \sm$ at the boundary $\partial([0,1]^2)$ and replace them with 
B-splines $N_{j_1,j_2}^{\ab{p}_2,\ab{r}_2}$ from the space $\mathcal{S}_h^{\ab{p}_2,\ab{r}_2}([0,1]^2)$ having a non-vanishing derivative of order $\ot \leq \sm$ at $\partial( [0,1]^2)$. To ensure the completeness of the resulting space we again have to add back
some of the 
B-splines from the space $\mathcal{S}_h^{\ab{p}_1,\ab{r}_1}([0,1]^2)$, but after first truncating them with respect to the space $\mathcal{S}_h^{\ab{p}_2,\ab{r}_2}([0,1]^2)$ in such a way that their derivatives of order $\ot \leq \sm$ vanish now at $\partial([0,1]^2)$. 
Consequently, we obtain a bivariate mixed degree and regularity spline space, denoted by $\mathcal{S}_h^{(\ab{p}_1, \ab{p}_2),(\ab{r}_1,\ab{r}_2)}([0,1]^2)$, via 
\begin{equation}  \label{eq:MixedSpace}
 \mathcal{S}_h^{(\ab{p}_1, \ab{p}_2),(\ab{r}_1,\ab{r}_2)}([0,1]^2) = \mathcal{S}_{1} ([0,1]^2)  \oplus \mathcal{\overline{S}}_1 ([0,1]^2) \oplus \mathcal{S}_2 ([0,1]^2), 
\end{equation}
with
\begin{align*}
& \mathcal{S}_{1} ([0,1]^2) =  \Span \left\{ 
N_{j_1,j_2}^{\ab{p}_1,\ab{r}_1}, \; j_1,j_2= \sm+1,\ldots, n_{1}-\sm-2\right\},\\[-0.7cm]
\end{align*}
\begin{align*}
& \mathcal{\overline{S}}_1 ([0,1]^2) = \Span \left\{ 
\overline{{N}}_{j_1,j_2}^{\,\ab{p}_1,\ab{r}_1}, \; j_1=1,\ldots,\sm, n_{1}-\sm-1, \ldots,n_{1}-2; \; j_2=1,\ldots,n_{1}-2 \right\} \oplus\\
& \qquad \qquad \Span \left\{ 
\overline{N}_{j_1,j_2}^{\,\ab{p}_1,\ab{r}_1}, \; j_1=\sm+1,\ldots,n_{1}-\sm-2; \; j_2=1,\ldots,\sm, n_{1}-\sm-1, \ldots,n_{1}-2 \right\}\\[-0.7cm]
\end{align*} 
and
\begin{align}
&\mathcal{S}_2 ([0,1]^2) = \Span \left\{ 
N_{j_1,j_2}^{\ab{p}_2,\ab{r}_2}, \; j_1=0,\ldots,\sm, n_{2}-\sm-1, \ldots,n_{2}-1; \; j_2=0,\ldots,n_{2}-1 \right\} \oplus \label{eq:space_s2} \\
 & \qquad \qquad \Span  \left\{ 
N_{j_1,j_2}^{\ab{p}_2,\ab{r}_2}, j_1=\sm+1,\ldots,n_{2}-\sm-2; \;  j_2 =  0,\ldots,\sm, n_{2}-\sm-1, \ldots,n_{2}-1 \right\}, \nonumber 
\end{align}
where $\overline{N}_{j_1,j_2}^{\,\ab{p}_1,\ab{r}_1} = \overline{N}_{j_1}^{\,{p}_1,{r}_1} \, \overline{N}_{j_2}^{\,{p}_1,{r}_1}$, and $\overline{N}_{i}^{\,p_1,r_1}$
is a truncation of ${N}_{i}^{p_1,r_1}$ defined in \eqref{eq:truncation}
, cf. Fig.~\ref{fig:SpacesMixed}.
Note that 
\begin{equation}  \label{eq:MixedSubSpacesRelations}
\mathcal{S}_1 ([0,1]^2)  \subseteq \mathcal{S}_h^{\ab{p}_1,\ab{r}_1}([0,1]^2) \quad {\rm and} \quad \overline{\mathcal{S}}_1 ([0,1]^2), \mathcal{S}_2 ([0,1]^2) \subseteq \mathcal{S}_h^{\ab{p}_2,\ab{r}_2}([0,1]^2),
\end{equation}
and that the underlying mixed degree and regularity spline space $\mathcal{S}_h^{(\ab{p}_1, \ab{p}_2),(\ab{r}_1,\ab{r}_2)}([0,1]^2)$ is uniquely determined by selecting the smoothness $\sm$ and degree $p_1$ via $\sm+1 \leq p_1 \leq 2\sm+1$ and by assumptions \eqref{eq:boundsOnR1andP1}.
\begin{figure}[htb!]
        \centering
    \includegraphics[scale=0.37]{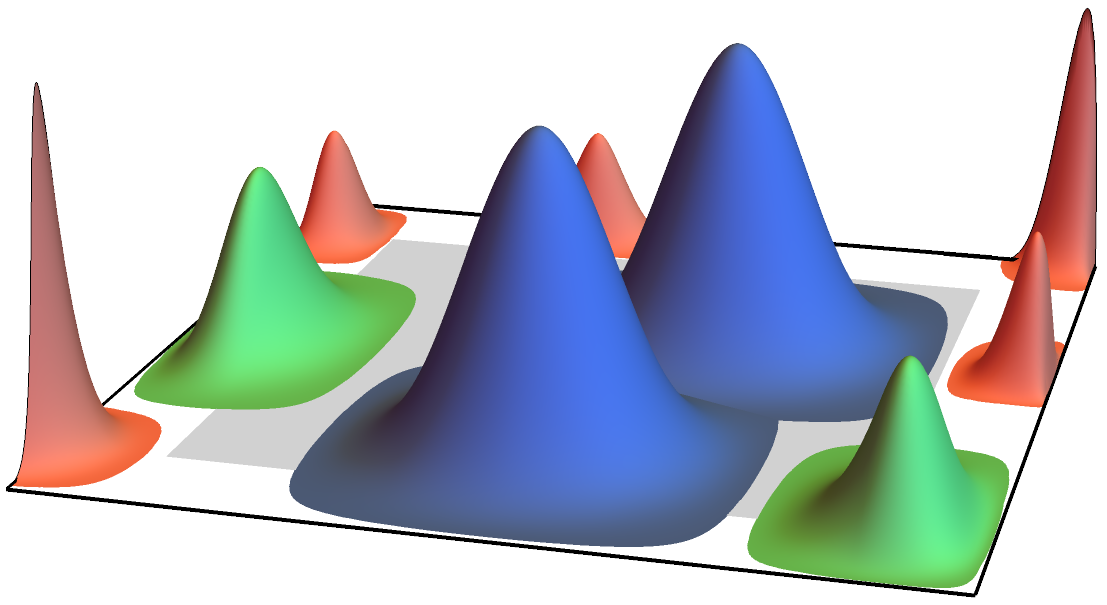}
    \hspace{1cm}
     \includegraphics[scale=0.37]{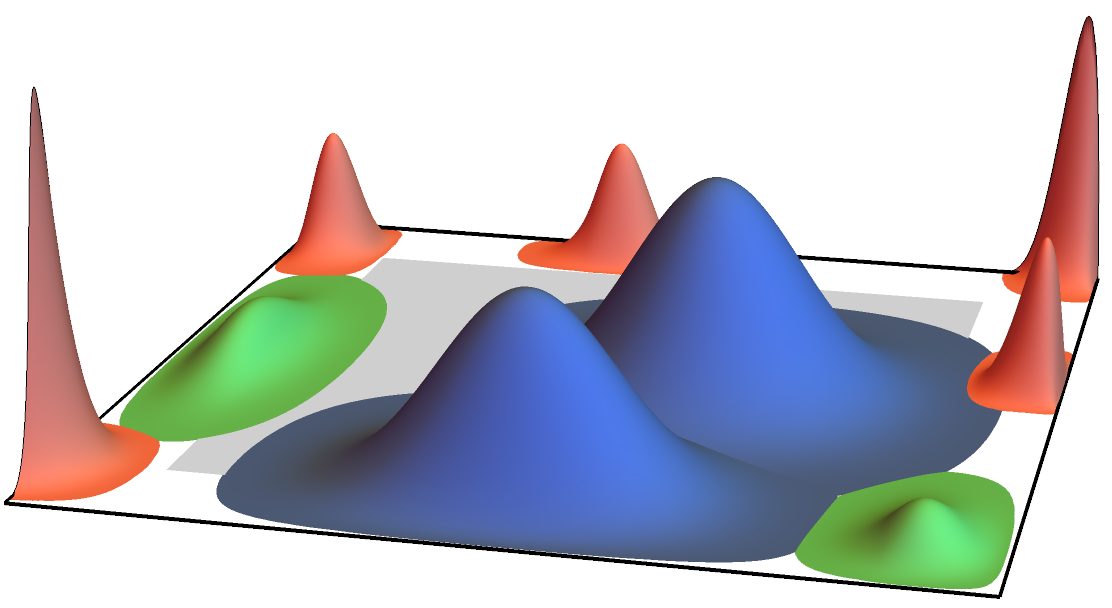}
    \vspace{0.4cm}
    
    \includegraphics[scale=0.4]{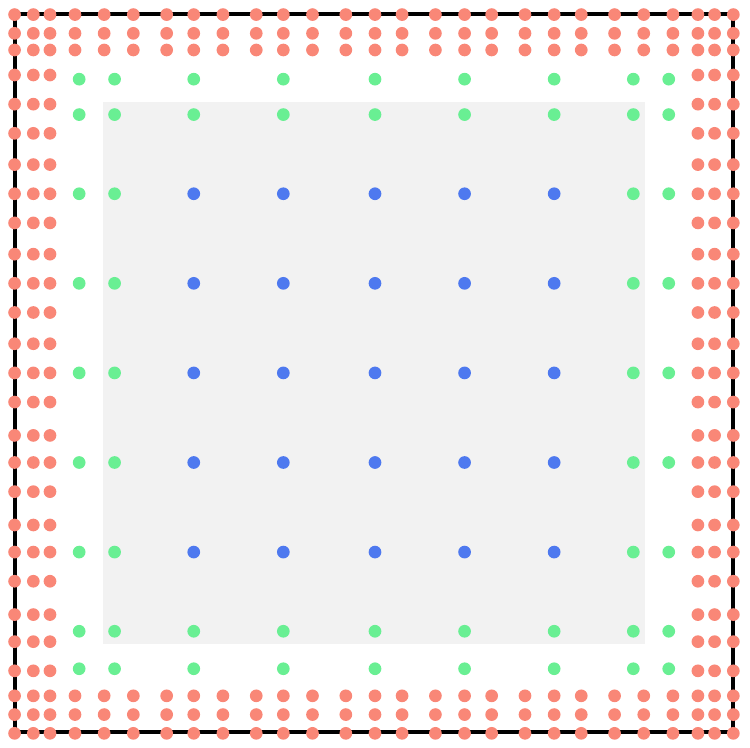}
    \hspace{2.9cm}
    \includegraphics[scale=0.4]{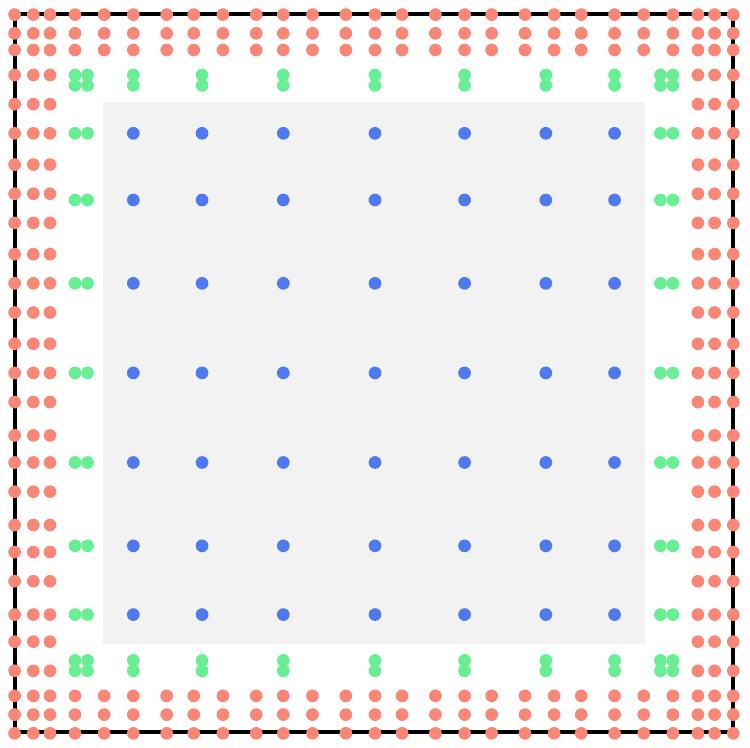}
    \caption{Instances of some basis functions (above) and the positions of extrema of all basis functions (below) of the spaces $\mathcal{S}_h^{(\ab{p}_1,\ab{p}_2),(\ab{r}_1,\ab{r}_2)}([0,1]^2)$ for $p_1=\sm+1, r_1=\sm$ (left) and $p_1=2\sm+1, r_1=2\sm$ (right) with $\sm=2$ and $k=7$. 
    The blue, green and red functions belong to the spaces $\mathcal{S}_1([0,1]^2)$, $\mathcal{\overline{S}}_1([0,1]^2)$ and $\mathcal{S}_2([0,1]^2)$, respectively. The gray region denotes the subdomain $[h,1-h]^2$.}
    \label{fig:SpacesMixed}
\end{figure}
Therefore, we get the two most practical and interesting cases by selecting the degree $p_1$ to be either the smallest or the greatest possible one:

\paragraph{Case A: $p_1=\sm+1$} Since by \eqref{eq:boundsOnR1andP1} we have $r_1=\sm$, and hence the regularities $r_1$ and $r_2$ are both equal to $\sm$, we can call the obtained underlying spline space also shortly just the {\em mixed degree spline space}, cf. Fig.~\ref{fig:SpacesMixed} (left). 

\paragraph{Case B: $p_1=2\sm+1$} Now the degrees $p_1$ and $p_2$ are both equal to $2\sm+1$, and therefore, we can call the obtained underlying spline space just the {\em mixed regularity spline space}, cf. Fig.~\ref{fig:SpacesMixed} (right).

\subsection{Properties of the mixed degree and regularity  spline spaces}

We will consider some properties of the obtained mixed degree and regularity spline spaces $\mathcal{S}_h^{(\ab{p}_1, \ab{p}_2),(\ab{r}_1, \ab{r}_2)}([0,1]^2)$. 

\begin{thm}
  Let the degrees ${p}_1,{p}_2$ and regularities ${r}_1,{r}_2$ be given as in \eqref{eq:boundsOnR1andP1}.
  All basis functions of the spaces $\mathcal{S}_{1} ([0,1]^2)$, $ \mathcal{\overline{S}}_1 ([0,1]^2)$ and $ \mathcal{S}_2 ([0,1]^2)$ are linearly independent and therefore form a basis 
  of the space $\mathcal{S}_h^{(\ab{p}_1, \ab{p}_2),(\ab{r}_1, \ab{r}_2)}([0,1]^2)$. Moreover, we have 
  $$
  \dim \mathcal{S}_h^{(\ab{p}_1, \ab{p}_2),(\ab{r}_1, \ab{r}_2)}([0,1]^2) = 4(\sm+1)^2(k+1) + (k+\sm)^2 + (p_1-(\sm+1))(p_1+2k+\sm-1).
  $$
\end{thm}
\begin{proof}
    Let us first prove the linear independency of the basis functions of the subspaces $\mathcal{S}_1([0,1]^2)$,  $\mathcal{\overline{S}}_1([0,1]^2)$ and $\mathcal{S}_2([0,1]^2)$. 
    Recall \eqref{eq:assumptionOnK}. 
    Restricted to the subdomain $[h,1-h]^2$, the basis functions from $\mathcal{\overline{S}}_1([0,1]^2)$ are actually functions from $\mathcal{S}_h^{\ab{p}_1,\ab{r}_1}([0,1]^2)$, since truncation only affects functions outside of the subdomain $[h,1-h]^2$. Therefore, the functions from $\mathcal{S}_1([0,1]^2)$ and $\mathcal{\overline{S}}_1([0,1]^2)$ are clearly linearly independent on the subdomain $[h,1-h]^2$ and consequently also on $[0,1]^2$. Additionally, the support of all functions from $\mathcal{S}_1([0,1]^2) \oplus\mathcal{\overline{S}}_1([0,1]^2)$ is at least partially contained in the subdomain $[h,1-h]^2$.
    Since on the other hand all functions from $\mathcal{S}_2([0,1]^2) $ have a vanishing support on the subdomain $[h,1-h]^2$ they are clearly linearly independent with functions from $\mathcal{S}_1([0,1]^2) \oplus\mathcal{\overline{S}}_1([0,1]^2)$. Also, they are linearly independent among them, since they are standard B-splines from $\mathcal{S}_h^{\ab{p}_2,\ab{r}_2}([0,1]^2)$.
    It remains to prove the dimension formula.
    It is straightforward to see that
    \begin{eqnarray*}
      \dim \mathcal{S}_1 ([0,1]^2)&=& (n_{1}-2(\sm+1))^2 = (p_1+1+k-2(\sm+1))^2,\\
      \dim \overline{\mathcal{S}}_1 ([0,1]^2)&=& 
      4 \sm (n_1-\sm-2) = 4 \sm (p_1+k -(\sm+1))\\
       \dim \mathcal{S}_2 ([0,1]^2)&=& 2 n_{2} (\sm+1) + 2(\sm+1)(n_{2}-2\sm-2) = 4(\sm+1)^2(k+1).
    \end{eqnarray*}
    Since $ \dim \mathcal{S}_h^{(\ab{p}_1, \ab{p}_2),(\ab{r}_1, \ab{r}_2)}([0,1]^2) = \dim \mathcal{S}_1([0,1]^2) + \dim \overline{\mathcal{S}}_1([0,1]^2) + \dim \mathcal{S}_2([0,1]^2) $, it follows
    \begin{align*}
    \dim \mathcal{S}_h^{(\ab{p}_1, \ab{p}_2),(\ab{r}_1, \ab{r}_2)}([0,1]^2) & \hspace{-0.1cm} = (p_1+1+k-2(\sm+1))^2 + 4 \sm (p_1+k -(\sm+1)) + 4 (\sm+1)^2(k+1) \\
    & \hspace{-0.1cm} = 4(\sm+1)^2(k+1) + (k+\sm)^2 + (p_1-(\sm+1))(p_1+2k+\sm-1).
    \end{align*}
\end{proof}
\begin{rem}
In the special cases, namely for {\em Case A} (i.e. where $p_1=\sm+1$) and {\em Case B} (i.e. where $p_1=2\sm +1$) the dimensions are
\begin{itemize}
\item {\em Case A}: $\dim \mathcal{S}_h^{(\ab{\sm}+\ab{1}, \ab{2\sm}+\ab{1}),(\ab{\sm}, \ab{\sm})}([0,1]^2) = 4(\sm+1)^2(k+1) + (k+\sm)^2.$  
\item {\em Case B}: $\dim \mathcal{S}_h^{(\ab{2\sm}+\ab{1}, \ab{2\sm}+\ab{1}),(\ab{2\sm}, \ab{\sm})}([0,1]^2) = 4(\sm+1)^2(k+1) + (k+\sm)^2 +s(3s+2k)$.
\end{itemize}
\end{rem}

\begin{thm}
   Let the degrees ${p}_1,{p}_2$ and regularities ${r}_1,{r}_2$ be given as in \eqref{eq:boundsOnR1andP1}. All basis functions of the space $\mathcal{S}_h^{(\ab{p}_1, \ab{p}_2),(\ab{r}_1, \ab{r}_2)}([0,1]^2)$ have a local support and form a convex partition of unity, i.e.~they are nonnegative and form a partition of unity. 
\end{thm}
\begin{proof} 
   We have to consider all basis functions forming the subspaces $\mathcal{S}_1([0,1]^2)$, $\mathcal{\overline{S}}_1([0,1]^2)$ and $\mathcal{S}_2([0,1]^2)$ defined in \eqref{eq:MixedSpace}.
   Clearly they have a local support, since (truncated) B-splines have a local support. 
   The nonnegativity follows directly by the nonnegativity of the B-splines and the definition of the truncation in \eqref{eq:truncation}. It remains to show that they form a partition of unity. 
   Let us define the two additional index sets
   \begin{align*}
     \widetilde{\mathcal{J}}_i = \{ \ab{j} \in \mathcal{J}_i, \; N_{\ab{j}}^{\ab{p}_i,\ab{r}_i} \in \mathcal{S}_i([0,1]^2) \} \subseteq \mathcal{J}_i, \quad i=1,2. 
   \end{align*}
   By starting with the partition of unity property for the B-splines of the space $\mathcal{S}_h^{\ab{p}_1,\ab{r}_1}([0,1]^2)$ and by using \eqref{eq:partitionMu_Bivariate}, we get
   \begin{align*}
   1 = & \sum_{\ab{j} \in \mathcal{J}_1} N_{\ab{j}}^{\ab{p}_1,\ab{r}_1} =  
   \sum_{\ab{j} \in \mathcal{J}_1} \sum_{\ab{\ell} \in \mathcal{J}_2} \mu_{\ab{\ell}}^{(\ab{j})} N_{\ab{\ell}}^{\ab{p}_2,\ab{r}_2} 
    =  \sum_{\ab{j} \in \mathcal{J}_1} \left( \sum_{\ab{\ell} \in \widetilde{\mathcal{J}}_2} \mu_{\ab{\ell}}^{(\ab{j})} N_{\ab{\ell}}^{\ab{p}_2,\ab{r}_2} + 
    \hspace{-0.15cm} \sum_{\ab{\ell} \in \mathcal{J}_2 \backslash \widetilde{\mathcal{J}}_2} \mu_{\ab{\ell}}^{(\ab{j})} N_{\ab{\ell}}^{\ab{p}_2,\ab{r}_2} \right)\\
    = & \sum_{\ab{\ell} \in \widetilde{\mathcal{J}}_2} \left(\sum_{\ab{j} \in \mathcal{J}_1}  \mu_{\ab{\ell}}^{(\ab{j})} \right) N_{\ab{\ell}}^{\ab{p}_2,\ab{r}_2} + 
    \sum_{\ab{j} \in \widetilde{\mathcal{J}}_1}  \sum_{\ab{\ell} \in \mathcal{J}_2\backslash \widetilde{\mathcal{J}}_2} \mu_{\ab{\ell}}^{(\ab{j})} N_{\ab{\ell}}^{\ab{p}_2,\ab{r}_2} + \hspace{-0.15cm} \sum_{\ab{j} \in \mathcal{J}_1 \backslash \widetilde{\mathcal{J}}_1}  \sum_{\;\ab{\ell} \in \mathcal{J}_2\backslash \widetilde{\mathcal{J}}_2} \mu_{\ab{\ell}}^{(\ab{j})} N_{\ab{\ell}}^{\ab{p}_2,\ab{r}_2} \\
    = & \sum_{\ab{\ell} \in \widetilde{\mathcal{J}}_2} N_{\ab{\ell}}^{\ab{p}_2,\ab{r}_2}  + \sum_{\ab{j} \in \widetilde{\mathcal{J}}_1}  \sum_{\;\ab{\ell} \in \mathcal{J}_2\backslash \widetilde{\mathcal{J}}_2} \mu_{\ab{\ell}}^{(\ab{j})} N_{\ab{\ell}}^{\ab{p}_2,\ab{r}_2} + \sum_{\ab{j} \in \mathcal{J}_1 \backslash \widetilde{\mathcal{J}}_1}  \sum_{\;\ab{\ell} \in \mathcal{J}_2\backslash \widetilde{\mathcal{J}}_2} \mu_{\ab{\ell}}^{(\ab{j})} N_{\ab{\ell}}^{\ab{p}_2,\ab{r}_2} \\
    = & 
     \sum_{\ab{j} \in \widetilde{\mathcal{J}}_2} N_{\ab{j}}^{\ab{p}_2,\ab{r}_2}  + \sum_{\ab{j} \in \widetilde{\mathcal{J}}_1} N_{\ab{j}}^{\ab{p}_1,\ab{r}_1} + \hspace{-0.15cm} \sum_{\ab{j} \in \mathcal{J}_1 \backslash\widetilde{\mathcal{J}}_1} \hspace{-0.15cm} \overline{N}_{\ab{j}}^{\,\ab{p}_1,\ab{r}_1}.
   \end{align*}
   By using \eqref{eq:truncation} and \eqref{eq:truncationPrecise}, the latter is exactly the sum of all functions from $\mathcal{S}_2([0,1]^2)$, $\mathcal{S}_1([0,1]^2)$ and $\overline{\mathcal{S}}_1([0,1]^2)$, and the proof is completed.

\end{proof}
\begin{thm}
    Let $\mathcal{S}_h^{(\ab{p}_1, \ab{p}_2),(\ab{r}_1, \ab{r}_2)}([0,1]^2)$ be defined as in \eqref{eq:MixedSpace} with the assumptions \eqref{eq:boundsOnR1andP1}. Then
    \begin{equation}  \label{eq:spaces}
        \mathcal{S}_h^{\ab{p}_1,\ab{r}_1}([0,1]^2) \subseteq \mathcal{S}_h^{(\ab{p}_1, \ab{p}_2),(\ab{r}_1, \ab{r}_2)}([0,1]^2) \subseteq \mathcal{S}_h^{\ab{p}_2,\ab{r}_2}([0,1]^2). 
    \end{equation}  
    Moreover, the space $\mathcal{S}_h^{(\ab{p}_1, \ab{p}_2),(\ab{r}_1, \ab{r}_2)}([0,1]^2)$ possesses the local reproduction property of polynomials of degree $\ab{p} \leq \ab{p}_1$.
\end{thm}
\begin{proof} 
    Let us start with the first relation in \eqref{eq:spaces}. Almost all functions from the space $\mathcal{S}_h^{\ab{p}_1,\ab{r}_1}([0,1]^2)\backslash $ 
    $\mathcal{S}_1([0,1]^2)$ can be represented with functions from  $\overline{\mathcal{S}}_1([0,1]^2)$ and $\mathcal{S}_2([0,1]^2)$ by re-truncating the functions from $\overline{\mathcal{S}}_1([0,1]^2)$. The only remaining functions are $N_{j_1,j_2}^{\ab{p}_1,\ab{r}_1}$ with $j_1 \in \{0, n_{1}-1\}$ or $j_2 \in \{0, n_{1}-1\}$. By symmetry it is enough to consider only the case $j_1=0$. By \eqref{eq:Bsplines_P1ToP2} one can see that 
    $$
    N_{0,j_2}^{\ab{p}_1,\ab{r}_1} = N_{0}^{{p}_1,{r}_1} N_{j_2}^{{p}_1,{r}_1} =
    \sum_{i=0}^\sm \mu_{i}^{(0)} N_{i}^{p_2,r_2} \sum_{j=0}^{n_{2}-1} {\mu}_{j}^{(j_2)} N_{j}^{p_2,r_2} = 
    \sum_{i=0}^\sm  \sum_{j=0}^{n_{2}-1} \mu_{i,j}^{(0,j_2)} N_{i,j}^{\ab{p}_2,\ab{r}_2}
    \in \mathcal{S}_2([0,1]^2).
    $$
    The second relation in \eqref{eq:spaces} follows directly from \eqref{eq:MixedSubSpacesRelations} and the fact that $\mathcal{S}_h^{\ab{p}_1,\ab{r}_1}([0,1]^2) \subseteq \mathcal{S}_h^{\ab{p}_2,\ab{r}_2}([0,1]^2)$.
    The local reproduction property of polynomials of degree $\ab{p} \leq \ab{p}_1$ is now a straightforward corollary of \eqref{eq:spaces}.
\end{proof}
When considering how close the mixed degree and regularity space $\mathcal{S}_h^{(\ab{p}_1, \ab{p}_2),(\ab{r}_1, \ab{r}_2)}([0,1]^2)$ is to both $\mathcal{S}_h^{\ab{p}_1, \ab{r}_1}([0,1]^2)$ and $\mathcal{S}_h^{\ab{p}_2, \ab{r}_2}([0,1]^2)$, we observe the following relations
$$
\lim_{h \to 0} \frac{\dim \mathcal{S}_h^{(\ab{p}_1, \ab{p}_2),(\ab{r}_1, \ab{r}_2)}([0,1]^2)}{\dim \mathcal{S}_h^{\ab{p}_1, \ab{r}_1}([0,1]^2)} = 1, \quad
\lim_{h \to 0} \frac{\dim \mathcal{S}_h^{\ab{p}_2, \ab{r}_2}([0,1]^2)}{\dim \mathcal{S}_h^{(\ab{p}_1, \ab{p}_2),(\ab{r}_1, \ab{r}_2)}([0,1]^2)} = \frac{(p_2-r_2)^2}{(p_1-r_1)^2}=(\sm+1)^2.
$$


\section{A $C^{\sS}$-smooth isogeometric multi-patch spline space of mixed degree and regularity} \label{sec:problem_statement}
 
We will present a novel $C^{\sS}$-smooth mixed degree and regularity isogeometric spline space over bilinearly parameterized planar multi-patch domains whose construction will be based on the underlying mixed degree and regularity spline space described in Section~\ref{subsec:underlyingSpace}, and will be motivated by the design of the $C^s$-smooth isogeometric spline space~\cite{KaVi20b} which possesses the same degree~$\ab{p}_2$ and inner patch regularity~$\ab{r}_2 \geq \ab{\sm}$ on the entire multi-patch domain and requires for the degree $p_2$ at least $p_2 = 2s+1$. The generated $C^{\sS}$-smooth isogeometric spline space of mixed degree~$(\ab{p}_1,\ab{p}_2)$ and mixed regularity~$(\ab{r}_1,\ab{r}_2)$ will keep the degree $p_2=2s+1$ and regularity $r_2=\sm$ 
in the vicinity of the edges and vertices and will 
have the degree $p_1$ and regularity $r_1=p_1-1$ in the interior of the patches of the multi-patch domain. 
Before presenting the construction of the $C^{\sS}$-smooth mixed degree and regularity isogeometric spline space we will first describe the configuration of the considered bilinear planar multi-patch parameterizations with their linear gluing functions and will further introduce 
$C^{\sS}$-smooth isogeometric spline functions (of mixed degree and regularity) over these multi-patch domains in a more general way.


 \subsection{The bilinearly parameterized multi-patch domain $\&$ the linear gluing functions} \label{subsec:multipatch}

Let $\Omega \subset \R^2$ be an open and connected planar domain, whose closure $\overline{\Omega}$ 
is a
disjoint union of open quadrilateral patches~$\Omega^{(i)}$, $i \in \mathcal{I}_{\Omega}$, of open edges~$\Gamma^{(i)}$, $i \in \mathcal{I}_{\Gamma}$, and of vertices~$\bfm{\Xi}^{(i)}$, $i \in \mathcal{I}_{\Xi}$, i.e.
\begin{equation*} 
\displaystyle
\overline{\Omega} = \bigcup_{i \in \mathcal{I}_{\Omega}} \Omega^{(i)}  \; {\cup}  \bigcup_{i \in \mathcal{I}_{\Gamma}} \Gamma^{(i)} \; {\cup} \bigcup_{i \in \mathcal{I}_{\Xi}} \bfm{\Xi}^{(i)},
\end{equation*}
cf. Fig~\ref{fig:multipatchCase}, 
where $\mathcal{I}_{\Omega}$, $\mathcal{I}_{\Gamma}$ and $\mathcal{I}_{\Xi}$ are the index sets of the indices of the patches~$\Omega^{(i)}$, edges~$\Gamma^{(i)}$ and vertices $\bfm{\Xi}^{(i)}$, respectively. Additionally we assume that the intersection of the closures of any two patches 
is either a single common edge, a single common vertex or is empty. Moreover, we divide $\mathcal{I}_{\Gamma}$ into $\mathcal{I}_{\Gamma}^I \dot{\cup} \mathcal{I}_{\Gamma}^B$ and $\mathcal{I}_{\Xi}$ into $ \mathcal{I}_{\Xi}^I \dot{\cup} \mathcal{I}_{\Xi}^B$, in order to separate the inner and boundary case of an edge and vertex.
The closure of each patch $\Omega^{(i)}$ is determined by four vertices $\bfm{\Xi}^{(i_j)}$, $j=1,2,3,4$, via a bilinear, bijective and regular geometry mapping~$\ab{F}^{(i)}: [0,1]^{2}  \rightarrow \R^{2}$, 
\begin{align*}
 \ab{F}^{(i)}(\bb{\xi}) = \ab{F}^{(i)}(\xi_1,\xi_2) = (1-\xi_1)(1-\xi_2) \bfm{\Xi}^{(i_1)} +
  \xi_1 (1-\xi_2) \bfm{\Xi}^{(i_2)} +
 (1-\xi_1) \xi_2 \,\bfm{\Xi}^{(i_3)} + 
  \xi_1 \xi_2 \,\bfm{\Xi}^{(i_4)}.
\end{align*}

\begin{figure}[bth]
    \centering
    \includegraphics[scale=0.25]{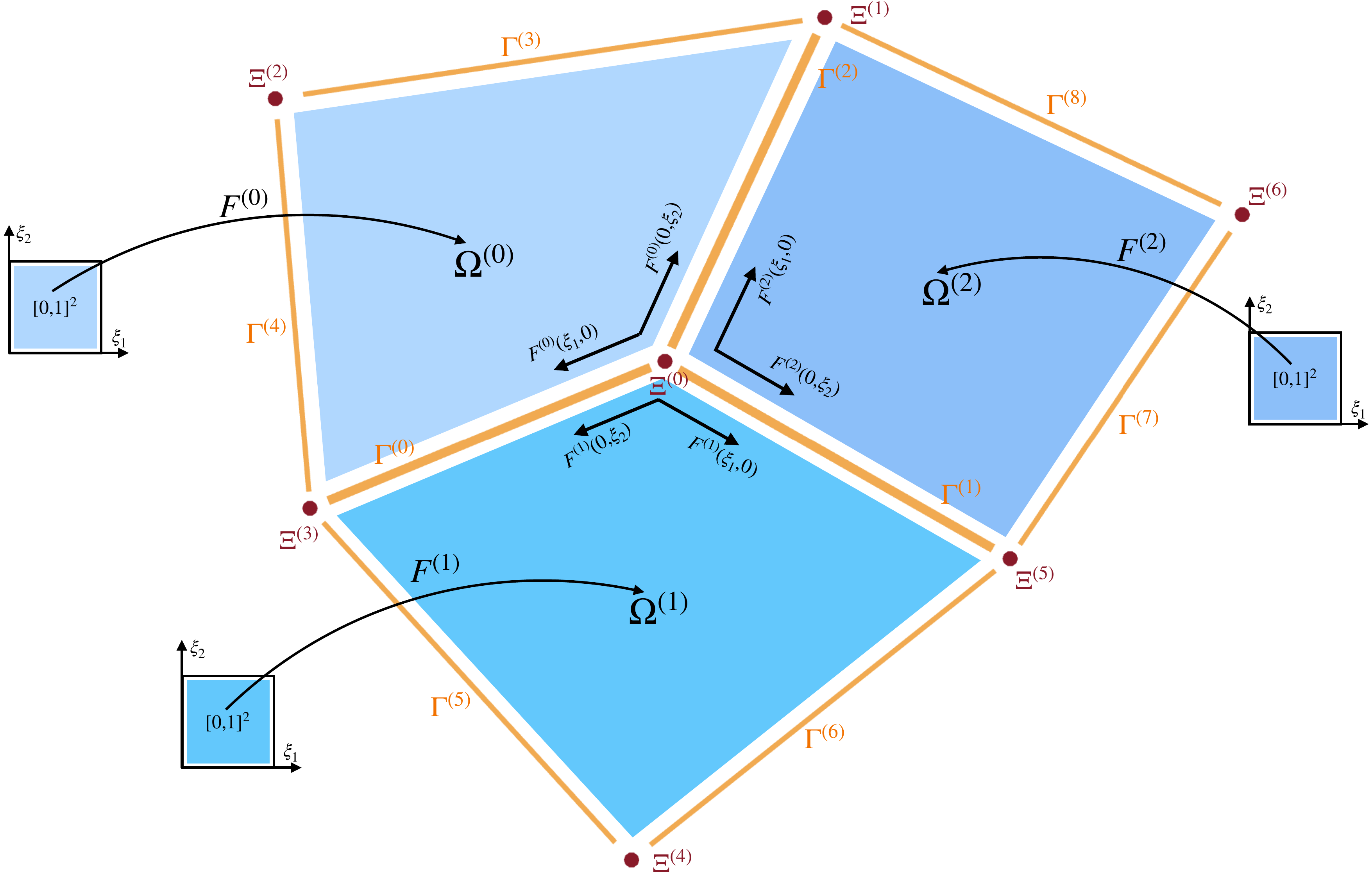}
    \caption{A particular example of domain $\overline{\Omega}$ with the open patches $\Omega^{(i)}$, defined by respective geometry mappings $\ab{F}^{(i)}$ (blue color), open edges $\Gamma^{(i)}$ (orange color) and the vertices $\bfm{\Xi}^{(i)}$ (red color).}
    \label{fig:multipatchCase}
\end{figure}
For each inner edge~$\Gamma^{(i)}$, $i \in \mathcal{I}_{\Gamma}^I$, where $\overline{\Gamma^{(i)}}=\overline{\Omega^{(i_0)}} \cap \overline{\Omega^{(i_1)}}$, $i_0,i_1 \in \mathcal{I}_{\Omega}$, we can define two linear polynomials $\alpha^{(i,\Side)},\beta^{(i,\Side)}, \Side \in \{\LL,\RR \}$, which are called gluing functions. 
We can assume that $\overline{\Gamma^{(i)}}$ is parameterized as 
$$
\ab{F}^{(i_0)}(0,\xi) = \ab{F}^{(i_1)}(0,\xi), \quad \xi \in [0,1],
$$
cf. Fig.~\ref{fig:twopatchCase}, which implies the gluing functions of the form
\begin{align}  \label{eq:alphaLRbar}
 \alpha^{(i,\Side)}(\xi) & = \lambda \det J \ab{F}^{(\Side)}(0,\xi), 
 \quad \lambda > 0, \nonumber\\
 &\\[-0.3cm]
 \beta^{(i,\Side)}(\xi) & = \frac{\partial_1 \ab{F}^{(\Side)}(0,\xi) \cdot \partial_2\ab{F}^{(\Side)}(0,\xi)}
 {||\partial_2 \ab{F}^{(\Side)}(0,\xi)||^{2}}{\new ,}\nonumber
\end{align}

where $J \ab{F}^{(\Side)}$, $\Side \in \{\LL,\RR \}$, are the Jacobian matrices of the geometry mappings $\ab{F}^{(\Side)}$. Since we supposed that the mappings $\ab{F}^{(\tau)}$, $\Side \in \{\LL,\RR \}$, are regular, it follows that $\alpha^{(i,i_0)} < 0$ and $\alpha^{(i,i_1)} > 0$. In the examples in Section~\ref{section_numerical_examples_Galerkin}, we select $\lambda$ in \eqref{eq:alphaLRbar} which minimizes
$
 || \alpha^{(i,\LL)}+1 ||^2_{L^2} + || \alpha^{(i,\RR)}-1 ||^2_{L^2},
$
cf. \cite{KaVi19a}.

\begin{figure}
    \centering
    \includegraphics[scale=0.25]{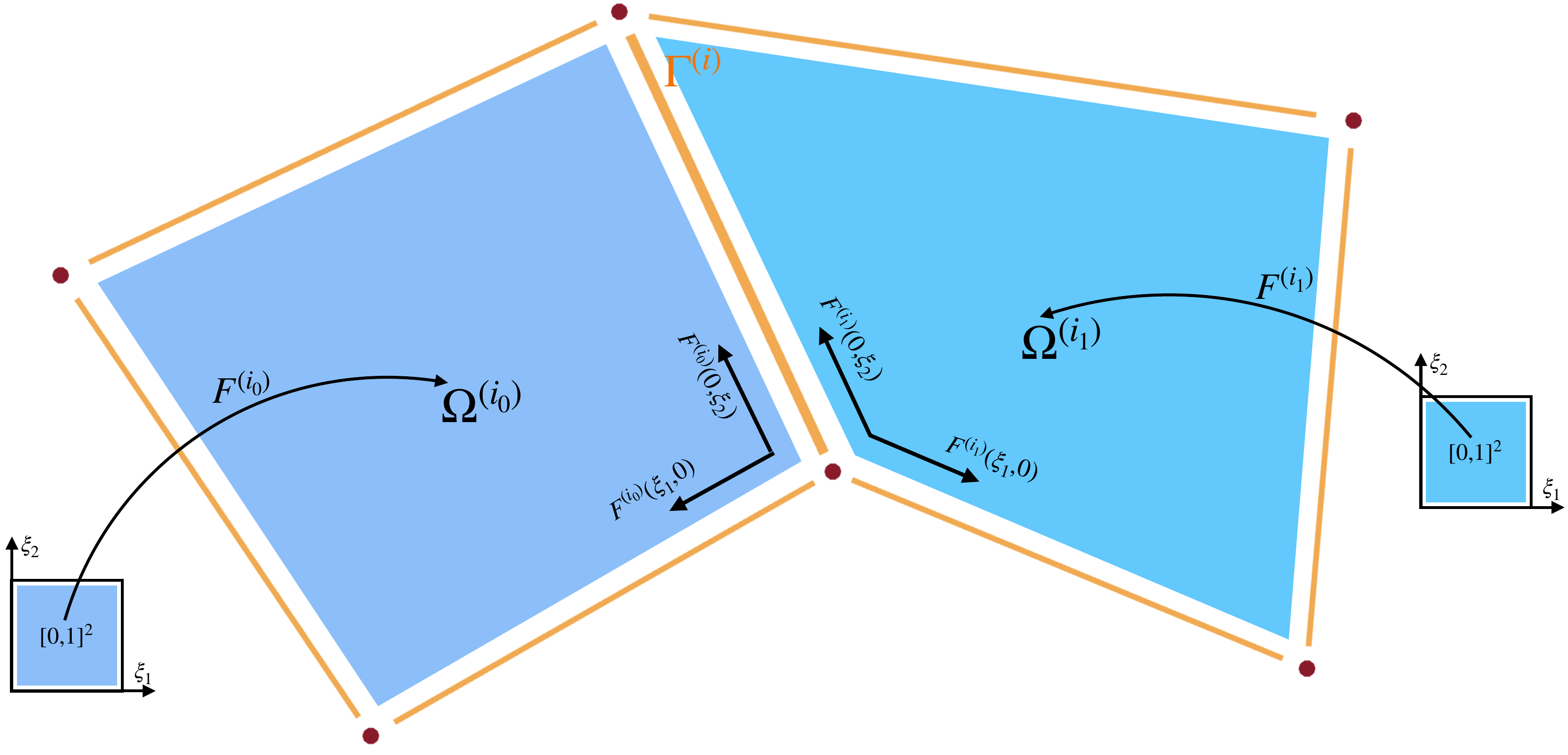}
    \caption{The parameterization of the two-patch domain $\overline{\Omega^{(i_0)} \cup \Omega^{(i_1)}}$ with the common edge $\Gamma^{(i)}$ and the 
    associated geometry mappings $\ab{F}^{(i_0)}$ and $\ab{F}^{(i_1)}$. }
    \label{fig:twopatchCase}
\end{figure}

\subsection{$C^{\sS}$-smooth isogeometric spline functions of mixed degree and regularity}


Before we introduce the space of $C^{\sS}$-smooth isogeometric spline functions of mixed degree~$(\ab{p}_1,\ab{p}_2)$ and mixed regularity~$(\ab{r}_1,\ab{r}_2)$ over bilinearly parameterized multi-patch domains, we first recall the 
{idea} of $C^{\sS}$-smooth isogeometric spline functions of the same degree~$\ab{p}_2$ and regularity $\ab{r}_2 = \ab{\sm}$ everywhere over bilinear multi-patch domains studied in~\cite{KaVi20b}. 

Let $\widetilde{\V}$ be the space of isogeometric spline functions of degree~$\ab{p}_2$ and regularity~$\ab{r}_2 = \ab{\sm}$ over the multi-patch domain~$\overline{\Omega}$, i.e.
\[
\widetilde{\V} = \left\{ \phi \in L^2(\overline{\Omega}): \; \phi \circ \ab{F}^{(i)}  \in {\mathcal{S}_{h}^{\ab{p}_2,\ab{r}_2}([0,1]^{2})}, \; 
i \in \mathcal{I}_{\Omega}   \right\} .
\]
For a function $\phi \in \widetilde{\V}$, we define for each inner edge~$\Gamma^{(i)}$, $i \in \mathcal{I}_{\Gamma}^I$, with $\overline{\Gamma^{(i)}}=\overline{\Omega^{(i_0)}} \cap \overline{\Omega^{(i_1)}}$, $i_0,i_1 \in \mathcal{I}_{\Omega}$, 
assuming that the parameterization of 
$\overline{\Gamma^{(i)}}$ is given as $\ab{F}^{(i_0)}(0,\xi) = \ab{F}^{(i_1)}(0,\xi)$, $\xi \in [0,1]$, cf. Fig.~\ref{fig:twopatchCase},
the functions
\begin{equation*}   \label{eq:gC2}
 \g_\ell^{(i,\Side)}(\xi) = \left(\alpha^{(i,\Side)}(\xi)\right)^{-\ell}\, \partial_1^\ell \g^{(\Side)}(0,\xi) - \sum_{j=0}^{\ell-1} {\ell \choose j} 
 \left(\frac{\beta^{(i,\Side)}(\xi)}{\alpha^{(i,\Side)}(\xi)}\right)^{\ell-j}  \dd^{\ell-j} \gC_j^{(i,\Side)}(\xi)
 \end{equation*}
for $\ell=0,\ldots, \sS$, and $\Side\in \{\LL,\RR\}$, where the functions $f^{(\Side)}$, $\Side \in \{\LL,\RR \}$, are the spline functions $f^{(\Side)}=\phi \circ \ab{F}^{(\Side)}$. Then, the space of $C^s$-smooth isogeometric spline functions of degree~$\ab{p}_2$ and regularity $\ab{r}_2$ over the multi-patch domain~$\overline{\Omega}$, which is given as
\begin{equation*}
\widetilde{\V}^{\sS} = \widetilde{\V} \cap C^{\sS}(\overline{\Omega}),
\end{equation*}
can be described as
\begin{equation} \label{eq:space_same_degree}
\widetilde{\V}^{\sS} = \left\{ \begin{array}{ll} \phi \in L^2(\overline{\Omega}): \; & \phi \circ \ab{F}^{(i)} \in {\mathcal{S}_h^{\ab{p}_2,\ab{r}_2}([0,1]^{2})},  \; 
i \in \mathcal{I}_{\Omega}, \mbox{ and } \\ & \g_\ell^{(i,\LL)}(\xi) = \g_\ell^{(i,\RR)}(\xi), \; \xi \in [0,1], \; \ell=0,\ldots, \sS,\; i \in \mathcal{I}_{\Gamma}^I \end{array}   \right\},
\end{equation}
cf.~\cite[Theorem 2]{KaVi20b}. 
For a function $ \phi \in \widetilde{\V}^{\sS}$, we denote for each inner edge~$\Gamma^{(i)}$, $i \in \mathcal{I}_{\Gamma}^I$, the equivalued terms
$\g_\ell^{(i,\LL)}(\xi) = \g_\ell^{(i,\RR)}(\xi)$, $\ell =0,\ldots,\sS$, by $\gC_{\ell}^{(i)}$, where the functions $\gC_{\ell}^{(i)}$ describe specific derivatives of order~$\ell$ of $\phi$ across the inner edge~$\Gamma^{(i)}$, cf.~\cite{KaVi20b}. In case of a boundary edge $\Gamma^{(i)}$, $i \in \mathcal{I}_{\Gamma}^{B}$, with $\overline{\Gamma^{(i)}} \subseteq \overline{\Omega^{(i_0)}}$, $i_0 \in \mathcal{I}_{\Omega}$, assuming that the boundary edge $\overline{\Gamma^{(i)}}$ is given by $\ab{F}^{(i_0)}(\{0 \} \times [0,1]) $, the functions $\gC_{\ell}^{(i)}$, $\ell =0,\ldots,\sS$, simply define the derivatives $\partial_1^{\ell} \left(\phi \circ \ab{F}^{(i_0)}(0,\xi)\right)$ of order $\ell$ across the boundary edge $\Gamma^{(i)}$ of a given function $ \phi \in \widetilde{\V}^{\sS}$.

We are now interested in the space of $C^s$-smooth isogeometric spline functions of mixed degree~$(\ab{p}_1,\ab{p}_2)$ and mixed regularity~$(\ab{r}_1,\ab{r}_2)$ over the multi-patch domain~$\overline{\Omega}$, which is given as
\[
\V^{\sS} = \V \cap C^{\sS}(\overline{\Omega})
\]
with
\begin{equation*}
\V = \left\{ \phi \in L^2(\overline{\Omega}): \; \phi \circ \ab{F}^{(i)} \in {\mathcal{S}_{h}^{(\ab{p}_1,\ab{p}_2),(\ab{r}_1,\ab{r}_2)}([0,1]^{2})}, \; 
i \in \mathcal{I}_{\Omega}   \right\}.
\end{equation*}
Due to $\mathcal{S}_{h}^{(\ab{p}_1,\ab{p}_2),(\ab{r}_1,\ab{r}_2)}([0,1]^2) \subseteq \mathcal{S}_h^{\ab{p}_2,\ab{r}_2}([0,1]^{2})$, and hence $\V \subseteq \widetilde{\V}$ and $\V^{\sS} \subseteq \widetilde{\V}^{\sS}$, the space~$\V^{\sS}$ can be represented similar to~\eqref{eq:space_same_degree} as
\begin{equation*} 
\V^{\sS} = \left\{ \begin{array}{ll} \phi \in L^2(\overline{\Omega}): \; & \phi \circ \ab{F}^{(i)} \in \mathcal{S}_{h}^{(\ab{p}_1,\ab{p}_2),(\ab{r}_1,\ab{r}_2)}([0,1]^2), \; 
i \in \mathcal{I}_{\Omega}, \mbox{ and } \\ & \g_\ell^{(i,\LL)}(\xi) = \g_\ell^{(i,\RR)}(\xi), \; \xi \in [0,1], \; \ell=0,\ldots, \sS,\; i \in \mathcal{I}_{\Gamma}^I \end{array}   \right\}.
\end{equation*}

\subsection{The $C^{\sS}$-smooth mixed degree and regularity isogeometric subspace} \label{subsec:Cs_space}

Inspired by the $C^{\sS}$-smooth isogeometric subspace
\begin{equation*}
\widetilde{\W}^{\sS} = \left\{ \phi \in \widetilde{\V}^{\sS}: \; \gC^{(i)}_\ell \in \mathcal{S}_h^{p_2-\ell,2\sm -\ell}([0,1]) , \; \ell=0,\ldots, \sS,\; i \in \mathcal{I}_{\Gamma}^I    \right\},
\end{equation*}
of the $C^{\sS}$-smooth isogeometric spline space~$\widetilde{\V}^{\sS}$ of degree~$\ab{p}_2$, introduced in~\cite{KaVi20b}, we generate instead of the entire $C^{\sS}$-smooth isogeometric spline space~$\V^{\sS}$ of mixed degree~$(\ab{p}_1,\ab{p}_2)$ and mixed regularity~$(\ab{r}_1,\ab{r}_2)$, which possesses a complex structure, 
the simpler $C^{\sS}$-smooth subspace
\begin{equation*}
\W^{\sS} = \left\{ \begin{array}{ll} \phi \in \V^{s}: \;  \gC^{(i)}_\ell \in \mathcal{S}_h^{p_2-\ell,2\sm -\ell}([0,1]) , \; \ell=0,\ldots, \sS,\; i \in \mathcal{I}_{\Gamma}^I \end{array}   \right\}.
\end{equation*}
Similar to $\widetilde{\W}^{\sS}$ in \cite{KaVi20b}, the $C^{\sS}$-smooth mixed degree and regularity subspace~$\W^{\sS}$ 
{can be constructed as the direct sum of smaller subspaces associated with the single patches~$\Omega^{(i)}$, $i \in \mathcal{I}_\Omega$, edges~$\Gamma^{(i)}$, $i \in \mathcal{I}_\Gamma$, and vertices $\bfm{\Xi}^{(i)}$, $i \in \mathcal{I}_\Xi$,}
namely as
\begin{equation} \label{eq:direct_sum_subspace}
\W^{\sS} = \left(\bigoplus_{i \in \mathcal{I}_\Omega} \mathcal{W}^\sS_{\Omega^{(i)}}\right) 
    \oplus \left(\bigoplus_{i \in \mathcal{I}_\Gamma} \mathcal{W}^\sS_{\Gamma^{(i)}}\right) 
    \oplus \left(\bigoplus_{i \in \mathcal{I}_\Xi} \mathcal{W}^\sS_{\bfm{\Xi}^{(i)}}\right) .
\end{equation}
In the following, we will present the derivation of the direct sum~\eqref{eq:direct_sum_subspace}, and will describe the design of the single subspaces.

Due to the direct sum~\eqref{eq:MixedSpace}, and the fact that for the $C^s$-smoothness conditions $\g_\ell^{(i,\LL)}(\xi) = \g_\ell^{(i,\RR)}(\xi), \; \xi \in [0,1], \; \ell=0,\ldots, \sS,\; i \in \mathcal{I}_{\Gamma}^I$, just basis functions of the spline space $\mathcal{S}_2 ([0,1]^2)$ are involved for the spline functions $f^{(\Side)}=\phi \circ \ab{F}^{(\Side)}$, $\Side \in \{\LL,\RR \}$, the space~$\W^{\sS}$ can be first decomposed into the direct sum
\[
\W^{\sS} = \W^{\sS}_{\Omega} \oplus \W^{\sS}_{\Gamma \cup \Xi},
\]
where 
\[
\W^{\sS}_{\Omega} =  \left\{ \phi \in L^2(\overline{\Omega}): \; \phi \circ \ab{F}^{(i)} \in {\mathcal{S}_{1} ([0,1]^2)  \oplus \mathcal{\overline{S}}_1 ([0,1]^2)} , \; 
i \in \mathcal{I}_{\Omega}   \right\}
\]
and
\begin{equation} \label{eq:WSGammaXi}
\W^{\sS}_{\Gamma \cup \Xi} =\left\{ \begin{array}{ll} \phi \in L^2(\overline{\Omega}): \; & \phi \circ \ab{F}^{(i)} \in {\mathcal{S}_{2} ([0,1]^2) }, \; 
i \in \mathcal{I}_{\Omega}, \mbox{ and } \\ 
& \g_\ell^{(i,\LL)}(\xi) = \g_\ell^{(i,\RR)}(\xi), \; \xi \in [0,1], \; \ell=0,\ldots, \sS,\; i \in \mathcal{I}_{\Gamma}^I, \mbox{ and } \\ 
&  \gC^{(i)}_\ell \in \mathcal{S}_h^{p_2-\ell,2\sm -\ell}([0,1]) , \; \ell=0,\ldots, \sS,\; i \in \mathcal{I}_{\Gamma}^I \end{array}   \right\}.
\end{equation}
The subspace $\W^{\sS}_{\Omega}$ can be further generated as the direct sum of the patch subspaces $\W^{\sS}_{\Omega^{(i)}}$, $i \in \mathcal{I}_{\Omega}$, i.e.
\[
\W^{\sS}_{\Omega} = \bigoplus_{i \in \mathcal{I}_{\Omega}} \W^{\sS}_{\Omega^{(i)}},
\]
with
\[
\W^{\sS}_{\Omega^{(i)}} =  \left\{ \phi \in L^2(\overline{\Omega}): \; \phi \circ \ab{F}^{(i)} \in {\mathcal{S}_{1} ([0,1]^2)  \oplus \mathcal{\overline{S}}_1 ([0,1]^2)} , \mbox{ and }  \phi \circ \ab{F}^{(j)} =0,  \; 
j \in \mathcal{I}_{\Omega} \setminus \{i \}    \right\}.
\]
Thereby, a possible basis of the subspace~$\W^{\sS}_{\Omega^{(i)}}$ is simply the set of those functions for which the corresponding function $\phi \circ \ab{F}^{(i)}$ is a basis function of the spline space $\mathcal{S}_{1} ([0,1]^2)  \oplus \mathcal{\overline{S}}_1 ([0,1]^2)$. The space~$\W^{\sS}_{\Gamma \cup \Xi}$ can be constructed as the direct sum of the edge subspaces $\mathcal{W}^\sS_{\Gamma^{(i)}}$, $i \in \mathcal{I}_{\Gamma}$, and of the vertex subspaces $\mathcal{W}^\sS_{\Xi^{(i)}}$, $i \in \mathcal{I}_{\Xi}$, i.e.
\[
\W^{\sS}_{\Gamma \cup \Xi} = \left(\bigoplus_{i \in \mathcal{I}_\Gamma} \mathcal{W}^\sS_{\Gamma^{(i)}}\right) 
    \oplus  
    \left(\bigoplus_{i \in \mathcal{I}_\Xi} \mathcal{W}^\sS_{\bfm{\Xi}^{(i)}}\right) .
\]
We first consider the case of the edge subspace~$\W^{\sS}_{\Gamma^{(i)}}$ for an edge~$\Gamma^{(i)}$, $i \in \mathcal{I}_{\Gamma}$, which is generated as
\begin{equation*}
\W^\sS_{\Gamma^{(i)}} = \left\{ \phi \in \W^{\sS}_{\Gamma \cup \Xi}: \;  \gC^{(i)}_\ell \in \mathcal{S}^{\ell}_{\Gamma^{(i)}}, \;
 \ell=0,\ldots, \sS, \mbox{ and } \gC^{(j)}_\ell =0 , \; \ell=0,\ldots, \sS,\; j \in \mathcal{I}_{\Gamma} \setminus \{ i\}    \right\},
\end{equation*}
where the spline spaces $\mathcal{S}^{\ell}_{\Gamma^{(i)}}$ are given as
\[
\mathcal{S}^{\ell}_{\Gamma^{(i)}} = \begin{cases}
                         \Span \left\{N_{j_2}^{p_2-\ell, 2\sm-\ell}| \; j_2=2\sm+1-\ell,
\ldots, n_{p_2-\ell, 2\sm-\ell}+ \ell -(2s+2) \right\} & \mbox{if }i \in \mathcal{I}_{\Gamma}^{I} \\
                         \Span \left\{ {N}_{j_2}^{\,p_2,r_2} |\; j_2=2\sm+1-\ell,
\ldots, n_2+ \ell -(2s+2) \right\} & \mbox{if }i \in \mathcal{I}_{\Gamma}^{B} 
                                \end{cases} .
\]
A possible basis of $\W^\sS_{\Gamma^{(i)}}$ is built by the set of those functions for which the corresponding functions $\gC^{(i)}_\ell$ are equal to a basis function of the spline space $\mathcal{S}^{\ell}_{\Gamma^{(i)}}$ for exactly one $\ell$, and equal to zero for all other $\ell$.

In case of the vertex subspace~$\W^{\sS}_{\bfm{\Xi}^{(i)}}$ for a vertex~$\bfm{\Xi}^{(i)}$, $i \in \mathcal{I}_{\Xi}$, with a patch valency~$\nu_i^{\Omega}$ and with an edge valency~$\nu_i^{\Gamma}$, assuming that all patches and edges around the vertex $\bfm{\Xi}^{(i)}$ are labeled in counterclockwise order as $\Omega^{(i_\rho)}$, $\rho=0,1,\ldots,v_i^{\Omega}-1$, and $\Gamma^{(i_\rho)}$, $\rho=0,1,\ldots,v_i^{\Gamma}-1$, respectively, and that the vertex~$\bfm{\Xi}^{(i)}$ is given as $\bfm{\Xi}^{(i)}=\ab{F}^{(i_0)}(0,0) = \ldots = \ab{F}^{(i_{\nu_i^{\Omega}-1})}(0,0)$, cf. Fig.~\ref{fig:multipatchCase}, the vertex subspace~$\W^{\sS}_{\bfm{\Xi}^{(i)}}$ is defined as
\begin{equation} \label{eq:WSXi}
\W^{\sS}_{\bfm{\Xi}^{(i)}} = \left\{ \begin{array}{ll} \phi \in \W^{\sS}_{\Gamma \cup \Xi}: \; & \gC^{(j)}_\ell \in \widetilde{\mathcal{S}}^{\ell}_{\Gamma^{(j)}}, \;
 \ell=0,\ldots, \sS, \; j \in \{i_0, \ldots, i_{\nu_i^\Gamma -1} \}, \mbox{ and } \\
 & \gC^{(j)}_\ell =0 , \; \ell=0,\ldots, \sS,\; j \in \mathcal{I}_{\Gamma} \setminus \{i_0, \ldots, i_{\nu_i^\Gamma -1} \}  \end{array} \right\},
\end{equation}
with the spline spaces
\[
\widetilde{\mathcal{S}}^{\ell}_{\Gamma^{(j)}} = \begin{cases}
                         \Span \left\{N_{j_2}^{p_2-\ell, 2\sm-\ell}| \; j_2=0, \ldots, 2\sm-\ell \right\} & \mbox{if }j \in \mathcal{I}_{\Gamma}^{I} \\
                         \Span \left\{ {N}_{j_2}^{\,p_2,r_2} |\; j_2=0, \ldots, 2\sm-\ell \right\} & \mbox{if }j \in \mathcal{I}_{\Gamma}^{B} 
                                \end{cases} .
\]
A possible basis of the space $\W^{\sS}_{\bfm{\Xi}^{(i)}}$ is given by a set of functions~$\phi$ for which the control points of the spline representations~$\phi \circ \ab{F}^{(i)} \in {\mathcal{S}_{2} ([0,1]^2) }$, $i \in \mathcal{I}_{\Omega}$, form a basis of the kernel of the homogeneous linear system for these control points which is determined by the stated conditions on the functions in the definition of the spaces $\W^{\sS}_{\Gamma \cup \Xi}$ and $\W^{\sS}_{\bfm{\Xi}^{(i)}}$ in~\eqref{eq:WSGammaXi} and \eqref{eq:WSXi}, respectively. In our examples in Section~\ref{section_numerical_examples_Galerkin} the basis of the kernel of the homogeneous linear system will be constructed by means of the concept of minimal determining sets, cf. \cite{KaVi17a,LaSch07}.   

Note that the resulting edge and vertex subspaces of the $C^s$-smooth mixed degree and regularity subspace~$\W^{\sS}$ coincide with the edge and vertex subspaces of the $C^\sS$-smooth space~$\widetilde{W}^{\sS}$ introduced in \cite{KaVi20b}. 
{This directly follows since}
the underlying spline space of the edge and vertex spaces for both $C^s$-smooth spaces $\W^{\sS}$ and $\widetilde{W}^{\sS}$ are the same, namely $\mathcal{S}_2 ([0,1]^2)$. For further details about the design of the edge and vertex subspaces of the $C^s$-smooth mixed degree and regularity subspace~$\W^{\sS}$, in particular about a possible basis construction for the single spaces, we hence refer to~\cite{KaVi20b}.

\section{The isogeometric Galerkin method} \label{sec:Galerkin}

We will present the isogeometric Galerkin method for solving the biharmonic equation with $C^1$-smooth and the triharmonic equation with $C^2$-smooth mixed degree and regularity isogeometric multi-patch spline functions over given planar multi-patch domains.  
\subsection{The isogeometric Galerkin method for 
the biharmonic equation} \label{subsec:Galerkin_biharmonic}

We first consider the isogeometric Galerkin approach for solving the biharmonic equation. 
Let $g:\overline{\Omega} \to \R$ 
and $g_1, g_2: \partial \Omega \to \R$ be sufficiently smooth functions. Then, we aim at 
computing $u:\overline{\Omega} \to \R$, 
which 
solves the biharmonic equation
\begin{align} \label{eq:biharmonic_problem_Galerkin}
 \triangle^{2} u(\ab{x})  & = g(\ab{x}), \,\,\quad  \ab{x} \in \overline{\Omega}, \nonumber \\[-0.25cm]
 & \\[-0.25cm]
u (\ab{x})  = g_1(\ab{x}) ,  \quad &
\partial_{\ab{n}} u (\ab{x})  = g_2(\ab{x}) ,   \quad \ab{x} \in \partial \Omega. \nonumber 
\end{align}
Using the weak formulation of \eqref{eq:biharmonic_problem_Galerkin}, we have to find 
$u = u_0 + u_g \in \mathcal{V}_{0} \oplus \mathcal{V}_{g}$, with 
\begin{align*}
 \mathcal{V}_{0} & = \{ v \in H^2(\overline{\Omega}) : 
 v(\ab{x}) = 0,\; 
\partial_{\ab{n}} v(\ab{x}) = 0 \mbox{ for }\ab{x} \in \partial \Omega \},\\
 \mathcal{V}_{g} & = \{ v \in H^2(\overline{\Omega}) : 
 v(\ab{x}) = g_1(x), \; 
\partial_{\ab{n}} v(\ab{x}) = g_2(x)  \mbox{ for }\ab{x} \in \partial \Omega \},
\end{align*}
such that for a fixed $u_g \in \mathcal{V}_{g}$ the equation
 \begin{equation} \label{eq:weak_biharmonic}
   \int_{\overline{\Omega}} \triangle u_0(\ab{x}) \triangle v_0(\ab{x}) \mathrm{d}\ab{x} = \int_{\overline{\Omega}} g(\ab{x}) v_0(\ab{x}) \mathrm{d}\ab{x} - 
   \int_{\overline{\Omega}} \triangle u_g(\ab{x}) \triangle v_0(\ab{x}) \mathrm{d}\ab{x},
 \end{equation}
 is satisfied  for all $v_0 \in \mathcal{V}_{0}$, cf. \cite{BaDe15, TaDe14, KaVi19a}. 
 To discretize the problem~\eqref{eq:weak_biharmonic} we employ the Galerkin projection and select as
 the finite dimensional discretization space the $C^1$-smooth mixed degree and regularity subspace $\mathcal{W}^1 \subseteq \mathcal{V}_{0} \oplus \mathcal{V}_{g}$ and denote the space $\mathcal{W}^1$ as $\mathcal{W}^1_h$ to emphasize the selected mesh size $h$. The subspace $\mathcal{W}^1_h$ is then first decomposed into $\mathcal{W}^1_{h,0} \oplus \mathcal{W}^1_{h,g}$,
 where
 $$ 
 \mathcal{W}^1_{h,0} = \{ v \in \mathcal{W}^1_h : 
 v(\ab{x}) = 0,\; 
\partial_{\ab{n}} v(\ab{x}) = 0 \mbox{ for }\ab{x} \in \partial \Omega \}
$$
 with a basis  $\{\phi_i\}_{i\in \mathcal{I}}$, $\mathcal{I} = \{0,1, \ldots, \dim \mathcal{W}_{h,0}^1-1 \}$. We now have to find the $C^1$-smooth approximation $u_{h} = u_{h,0} + u_{h,g}$, 
 where $u_{h,g}$ is first computed as the solution of the quadratic minimization problem 
 $$
 \min_{ u_{h,g} \in \mathcal{W}_{h,g}^1} \left( 
 \int_{\partial{\Omega}} ( u_{h,g}(\ab{x}) - g_1(\ab{x}))^2 \mathrm{d}\ab{x} + 
 \omega \int_{\partial{\Omega}} ( \partial_{\ab{n}} u_{h,g}(\ab{x}) - g_2(\ab{x}))^2 \mathrm{d}\ab{x} \right)
 $$
 for an appropriate 
 non-negative weight $\omega$, 
 while 
$ 
u_{h,0}(\ab{x})=\sum_{i \in \mathcal{I}}
 c_{i} \phi_i(\ab{x}) \in \mathcal{W}_{h,0}^1,  \,  c_{i} \in  \R,
$
is determined by solving the system of equations
\begin{equation*} 
  \int_{\overline{\Omega}}  \triangle u_{h,0}(\ab{x}) \triangle v_{h,0}(\ab{x}) \mathrm{d}\ab{x} = \int_{\overline{\Omega}} g(\ab{x}) v_{h,0}(\ab{x}) \mathrm{d}\ab{x} - 
  \int_{\overline{\Omega}}  \triangle u_{h,g}(\ab{x}) \triangle v_{h,0}(\ab{x}) \mathrm{d}\ab{x}
 \end{equation*}
for all $v_{h,0} \in  \W^1_{h,0}$. 
The obtained linear system $
S \ab{c}=\ab{g}
$, with $S = (s_{i,j} )_{i,j \in \mathcal{I}}$, $\ab{g} = (g_i)_{i \in \mathcal{I}}$, for the unknown coefficients $\ab{c} = (c_{i})_{i \in \mathcal{I}}$,
is given by
\begin{equation} \label{eq:sijB} 
 s_{i,j}=\int_{\overline{\Omega}} \triangle \phi_{i}(\ab{x}) \triangle \phi_{j}(\ab{x}) \mathrm{d}\ab{x} 
 \quad {\rm  and} \quad  
 g_{i}= \int_{\overline{\Omega}} g(\ab{x}) \phi_{i}(\ab{x}) 
 \;\mathrm{d}\ab{x}  - 
  \int_{\overline{\Omega}}  \triangle u_{h,g}(\ab{x}) \triangle \phi_{i}(\ab{x}) \mathrm{d}\ab{x}.
 \end{equation}
Employing the isogeometric formulation, we can compute the elements in \eqref{eq:sijB} 
by 
$
s_{i,j} = \sum_{\ell \in \mathcal{I}_{\Omega}} s^{(\ell)}_{i,j}
$
and
$
g_{i} = \sum_{\ell \in \mathcal{I}_{\Omega}} g^{(\ell)}_{i},
$
where
\begin{align*}  
    s^{(\ell)}_{i,j} & =  
    \int_{[0,1]^{2}}  \frac{1}{|\det J \ab{F}^{(\ell)}(\bb{\xi}^{})|}  \,
    \nabla \hspace{-0.08cm} \circ \hspace{-0.08cm} \left(  N^{(\ell)}(\bb{\xi}^{}) \, \nabla \left(\phi_i( \ab{F}^{(\ell)}(\bb{\xi}^{})) \right)\right) 
    \nabla \hspace{-0.08cm} \circ \hspace{-0.08cm} \left( N^{(\ell)}(\bb{\xi}^{}) \,\nabla \left(\phi_j( \ab{F}^{(\ell)}(\bb{\xi}^{}))\right)\right) 
  \mathrm{d}\bb{\xi}^{},\\
g^{(\ell)}_{i}  & = \int_{[0,1]^{2}} g(\ab{F}^{(\ell)}(\bb{\xi}^{})) \,\phi_i( \ab{F}^{(\ell)}(\bb{\xi}^{})) \,|\det J \ab{F}^{(\ell)}(\bb{\xi}^{})| \; 
\mathrm{d}\bb{\xi}^{}  \\
& - \int_{[0,1]^{2}}  \frac{1}{|\det J \ab{F}^{(\ell)}(\bb{\xi}^{})|}  \,
    \nabla \hspace{-0.08cm} \circ \hspace{-0.08cm} \left(  N^{(\ell)}(\bb{\xi}^{}) \, \nabla \left(u_{h,g}( \ab{F}^{(\ell)}(\bb{\xi}^{})) \right)\right) 
    \nabla \hspace{-0.08cm} \circ \hspace{-0.08cm} \left( N^{(\ell)}(\bb{\xi}^{}) \,\nabla \left(\phi_i( \ab{F}^{(\ell)}(\bb{\xi}^{}))\right)\right) 
  \mathrm{d}\bb{\xi}^{},
\end{align*}
with 
\begin{equation*}
N^{(\ell)}(\bb{\xi}) = \left|\det J \ab{F}^{(\ell)}(\bb{\xi})\right|  \left(J \ab{F}^{(\ell)}(\bb{\xi})\right)^{-T}  \left(J \ab{F}^{(\ell)}(\bb{\xi})\right)^{-1},
\end{equation*}
cf. \cite{BaDe15, KaVi17a}.

\subsection{The isogeometric Galerkin method for 
the triharmonic equation} \label{subsec:Galerkin_triharmonic}

Let us now also briefly study the isogeometric Galerkin method for 
the triharmonic equation.
Now, we have to compute $u:\overline{\Omega} \to \R$, 
which 
solves the triharmonic equation
\begin{align} \label{eq:triharmonic_problem}
 \triangle^{3} u(\ab{x})   & = - g(\ab{x}), \;\;  \ab{x} \in \overline{\Omega}, \nonumber \\[-0.25cm]
 & \\[-0.25cm]
u (\ab{x})  = g_1(\ab{x}) ,   \quad 
\partial_{\ab{n}} u (\ab{x}) & =  g_2(\ab{x}) ,  \quad 
\triangle u(\ab{x})  = g_3(\ab{x}) ,   \quad \ab{x} \in \partial \Omega, \nonumber
\end{align}
where $g, g_1, g_2$, $g_3$ are sufficiently smooth functions. 
The weak formulation of \eqref{eq:triharmonic_problem} requires to find $u = u_0 + u_g \in \mathcal{V}_{0} \oplus \mathcal{V}_{g}$, with 
\begin{align*}
 \mathcal{V}_{0}  & = \{ v \in H^3(\overline{\Omega}) : 
 v(\ab{x}) = 0,\; 
\partial_{\ab{n}} v(\ab{x}) = 0, \; \triangle v(\ab{x}) = 0 \mbox{ for }\ab{x} \in \partial \Omega \},\\
 \mathcal{V}_{g}  & = \{ v \in H^3(\overline{\Omega}) : 
 v(\ab{x}) = g_1(x), \; 
\partial_{\ab{n}} v(\ab{x}) = g_2(x), \;  \triangle v(\ab{x}) = g_3(x)  \mbox{ for }\ab{x} \in \partial \Omega \},
\end{align*}
such that for a fixed $u_g \in \mathcal{V}_{g}$ the equation
 \begin{equation*} \label{eq:weak_triharmonic}
   \int_{\overline{\Omega}} \nabla\left(\triangle u_0(\ab{x})\right) \circ \nabla \left( \triangle v_0(\ab{x}) \right) \mathrm{d}\ab{x} = \int_{\overline{\Omega}} g(\ab{x}) v_0(\ab{x}) \mathrm{d}\ab{x} - 
   \int_{\overline{\Omega}} \nabla\left(\triangle u_g(\ab{x})\right) \circ \nabla \left( \triangle v_0(\ab{x}) \right) \mathrm{d}\ab{x},
 \end{equation*}
 is satisfied  for all $v_0 \in \mathcal{V}_{0}$, cf. \cite{BaDe15, TaDe14, KaVi19a}. 
For the Galerkin projection we choose now the $C^2$-smooth mixed degree and regularity subspace $\mathcal{W}^2 \subseteq \mathcal{V}_{0} \oplus \mathcal{V}_{g}, $ 
and denote it again as $\mathcal{W}^2_h = \mathcal{W}^2_{h,0} \oplus \mathcal{W}^2_{h,g}$,
  with
$
\mathcal{W}^2_{h,0} = \mathcal{W}^2_h \cap \mathcal{V}_{0},
$
having a basis  $\{\phi_i\}_{i\in \mathcal{I}}$, $\mathcal{I} = \{0,1, \ldots, \dim \mathcal{W}_{h,0}^2-1 \}$.
 We have to find the $C^2$-smooth approximation $u_{h} = u_{h,0} + u_{h,g}$ by first computing
 the non-homogeneous part $u_{h,g} \in \mathcal{W}_{h,g}^2$ via minimizing the expression
 \begin{equation*}
 \int_{\partial{\Omega}} ( u_{h,g}(\ab{x}) - g_1(\ab{x}))^2 \mathrm{d}\ab{x} + 
 \omega_1 \int_{\partial{\Omega}} ( \partial_{\ab{n}} u_{h,g}(\ab{x}) - g_2(\ab{x}))^2 \mathrm{d}\ab{x} 
 + \, \omega_2 \int_{\partial{\Omega}} ( \triangle u_{h,g}(\ab{x}) - g_3(\ab{x}))^2 \mathrm{d}\ab{x} 
 \end{equation*}
 for suitable non-negative weights $\omega_1$, $\omega_2$, and then constructing the homogeneous part 
$ 
u_{h,0}(\ab{x})=\sum_{i \in \mathcal{I}}
 c_{i} \phi_i(\ab{x}) \in \mathcal{W}_{h,0}^2
 $,  
 $c_{i} \in  \R,$
 by solving the system of equations
\begin{equation*} 
  \int_{\overline{\Omega}}  \nabla\left(\triangle u_{h,0}(\ab{x}) \right) \circ \nabla \left( \triangle v_{h,0}(\ab{x})\right) \mathrm{d}\ab{x} = \int_{\overline{\Omega}} g(\ab{x}) v_{h,0}(\ab{x}) 
  \mathrm{d}\ab{x}-
  \int_{\overline{\Omega}}  \nabla\left(\triangle u_{h,g}(\ab{x}) \right) \circ \nabla \left( \triangle v_{h,0}(\ab{x})\right) \mathrm{d}\ab{x}
 \end{equation*}
for all $v_{h,0} \in  \W^2_{h,0}$.
This implies again a linear system  
$
\widetilde{S}\ab{c}=\widetilde{\ab{g}}
$ with 
$\widetilde{S} = (\widetilde{s}_{i,j} )_{i,j \in \mathcal{I}}$, $\widetilde{\ab{g}} = (\widetilde{g}_i)_{i \in \mathcal{I}}$,
for the unknown coefficients $\ab{c} = (c_{i})_{i \in \mathcal{I}}$, where 
\begin{equation*} \label{eq:sijT} 
 \widetilde{s}_{i,j}=\int_{\overline{\Omega}} \nabla\left(\triangle \phi_{i}(\ab{x}) \right) \circ \nabla \left(\triangle \phi_{j}(\ab{x}) \right) \mathrm{d}\ab{x}, \;\;  
 \widetilde{g}_{i}= \int_{\overline{\Omega}} g(\ab{x}) \phi_{i}(\ab{x}) - 
 \int_{\overline{\Omega}} \nabla\left(\triangle u_{h,g}(\ab{x}) \right) \circ \nabla \left(\triangle \phi_{i}(\ab{x}) \right) \mathrm{d}\ab{x}.
 \end{equation*} 
Setting again
$  
\widetilde{s}_{i,j} = \sum_{\ell \in \mathcal{I}_{\Omega}} \widetilde{s}^{(\ell)}_{i,j},  \; \widetilde{g}_{i} = \sum_{\ell \in \mathcal{I}_{\Omega}} \widetilde{g}^{(\ell)}_{i}
$
and using the isogeometric formulation gives
\begin{align*}  
    \widetilde{s}^{(\ell)}_{i,j} & =  \int_{[0,1]^{2}} \nabla \left( \frac{1}{|\det J \ab{F}^{(\ell)}(\bb{\xi}^{})|}  \,
    \nabla \hspace{-0.08cm}\circ\hspace{-0.08cm} \left(  N^{(\ell)}(\bb{\xi}^{}) \, \nabla \left(\phi_i( \ab{F}^{(\ell)}(\bb{\xi}^{})) \right)\right) \right)
    \circ \nonumber \\
    & \; \left(   N^{(\ell)}(\bb{\xi}^{}) \,  \nabla \left( \frac{1}{|\det J \ab{F}^{(\ell)}(\bb{\xi}^{})|}  \,
    \nabla \hspace{-0.08cm} \circ \hspace{-0.08cm} \left( N^{(\ell)}(\bb{\xi}^{}) \,\nabla \left(\phi_j( \ab{F}^{(\ell)}(\bb{\xi}^{}))\right)\right) \right) \right)\mathrm{d}\bb{\xi}^{}, \nonumber\\[-1.1cm]
\end{align*}
and
\begin{align*}
\widetilde{g}^{(\ell)}_{i}  & = \int_{[0,1]^{2}} g(\ab{F}^{(\ell)}(\bb{\xi}^{})) \,\phi_i( \ab{F}^{(\ell)}(\bb{\xi}^{})) \,|\det J \ab{F}^{(\ell)}(\bb{\xi}^{})| \; 
\mathrm{d}\bb{\xi}^{} \\
& - \int_{[0,1]^{2}} \nabla \left( \frac{1}{|\det J \ab{F}^{(\ell)}(\bb{\xi}^{})|}  \,
    \nabla \hspace{-0.08cm}\circ\hspace{-0.08cm} \left(  N^{(\ell)}(\bb{\xi}^{}) \, \nabla \left(u_{h,g}( \ab{F}^{(\ell)}(\bb{\xi}^{})) \right)\right) \right)
    \circ \nonumber \\
    & \; \left(   N^{(\ell)}(\bb{\xi}^{}) \,  \nabla \left( \frac{1}{|\det J \ab{F}^{(\ell)}(\bb{\xi}^{})|}  \,
    \nabla \hspace{-0.08cm} \circ \hspace{-0.08cm} \left( N^{(\ell)}(\bb{\xi}^{}) \,\nabla \left(\phi_i( \ab{F}^{(\ell)}(\bb{\xi}^{}))\right)\right) \right) \right)\mathrm{d}\bb{\xi}^{},
\end{align*}
cf. \cite{BaDe15, KaVi17a, KaVi19a}.

\section{Numerical examples} \label{section_numerical_examples_Galerkin}

We 
illustrate by performing various examples
the potential of our proposed $C^s$-smooth mixed degree and regularity isogeometric spline space to solve the biharmonic equation \eqref{eq:biharmonic_problem_Galerkin} and the triharmonic equation~\eqref{eq:triharmonic_problem} over two different bilinearly parameterized domains and one bilinear-like $G^2$ multi-patch parameterizations via the isogeometric Galerkin approach. 
For this purpose, 
{we investigate the convergence  under $h$-refinement}
using  $C^\sm$-smooth mixed degree and regularity isogeometric spline spaces~$\W_h^\sm$ 
based on  the underlying mixed degree and regularity spline spaces $\mathcal{S}_h^{(\ab{p}_1, \ab{2\sm}+\ab{1}),(\ab{p}_1-\ab{1},\ab{\sm})}([0,1]^2)$, $\sm=1,2$. 


In all examples below, the functions $g, g_1, g_2$ determining the right-side of the biharmonic equation \eqref{eq:biharmonic_problem_Galerkin} and the functions $g, g_1, g_2, g_3$ describing the right-side of the triharmonic equation \eqref{eq:triharmonic_problem} shall be computed from the exact solution 
\begin{equation}  \label{eq:exactSolution}
 u(x_1,x_2)= \cos\left({x_1}\right) \sin\left({x_2}\right). 
\end{equation}
The obtained approximants $u_h \in \mathcal{W}^\sm_h$ will be compared with the exact solution \eqref{eq:exactSolution} by computing the relative errors with respect to the $L^2$-norm, $H^1$-seminorm and equivalents (cf. \cite{BaDe15})
of the $H^2$-seminorm and $H^3$-seminorm (in the case $s=2$), i.e.
\begin{equation} \label{eq:eqiuv2seminorms}
\frac{\| u-u_h\|_{L^2}}{\| u \|_{L^2}}, \quad
\frac{| u-u_h|_{H^1}}{| u |_{H^1}}, \quad
 \frac{\| \Delta u- \Delta u_h\|_{L^2}}{\| \Delta u \|_{L^2}}, \quad 
\frac{ \| \nabla (\Delta u) - \nabla(\Delta u_h)\|_{L^2}}{\| \nabla(\Delta u )\|_{L^2}}.
\end{equation}
We will call in the following examples the equivalents of the $H^m$-seminorms simply $H^m$-seminorms.


We now consider the three aforementioned planar 
multi-patch domains. The first one is the bilinear three-patch domain~$\overline{\Omega}$ shown in Fig.~\ref{fig:domains} (left), where the parameterizations of the individual patches $\overline{\Omega^{(i)}}$, $i=0,1,2$, are given as
\begin{equation*}  \label{eq:five_patch_domain_param}
\ab{F}^{(i)}(\xi_1,\xi_2) = \ab{\Xi}_0 (1-\xi_1) (1-\xi_2) + \ab{\Xi}_{2i+1} \,(1-\xi_1) \xi_2  + \ab{\Xi}_{2i+2} \,\xi_1 \xi_2 + \ab{\Xi}_{2i+3} \,\xi_1 (1-\xi_2),
\end{equation*}
with
\begin{equation} \label{eq:verticesThreePatch}
\begin{split}
\quad \ab{\Xi}_0  & = \begin{pmatrix}2 \\ 13/10 \end{pmatrix}, \; \ab{\Xi}_1 = \begin{pmatrix}21/10 \\ 7/10 \end{pmatrix},  \;\ab{\Xi}_2 = \begin{pmatrix}12/5 \\ 9/10 \end{pmatrix},  \;\ab{\Xi}_3 = \begin{pmatrix}5/2 \\ 8/5  \end{pmatrix},  
\\
\ab{\Xi}_4  & = \begin{pmatrix}37/20 \\ 19/10 \end{pmatrix},  \;\ab{\Xi}_5 = \begin{pmatrix}27/20 \\ 7/5 \end{pmatrix},\;
\ab{\Xi}_6  = \begin{pmatrix}17/10 \\ 4/5 \end{pmatrix}, \;
\ab{\Xi}_{7} := \ab{\Xi}_1. 
\end{split}
\end{equation}
The second one is the bilinear five-patch domain~$\overline{\Omega}$ visualized in Fig.~\ref{fig:domains} (middle), where again
the parameterizations of the individual patches $\overline{\Omega^{(i)}}$, $i=0,\ldots,4$, are
\begin{equation*}  \label{eq:five_patch_domain_param}
\ab{F}^{(i)}(\xi_1,\xi_2) = \ab{\Xi}_0 (1-\xi_1) (1-\xi_2) + \ab{\Xi}_{2i+1} \,(1-\xi_1) \xi_2  + \ab{\Xi}_{2i+2} \,\xi_1 \xi_2 + \ab{\Xi}_{2i+3} \,\xi_1 (1-\xi_2),
\end{equation*} 
with
\begin{align}  \label{eq:verticesFivePatch}
\ab{\Xi}_0 & = \begin{pmatrix}7 \\ 6 \end{pmatrix}, \, \ab{\Xi}_1 = \begin{pmatrix}57/5 \\ 6 \end{pmatrix},  \,\ab{\Xi}_2 = \begin{pmatrix}113/10 \\ 33/4 \end{pmatrix},  \,\ab{\Xi}_3 = \begin{pmatrix}93/10 \\ 199/20 \end{pmatrix},  \nonumber \\
\ab{\Xi}_4 & = \begin{pmatrix}31/5 \\ 45/4  \end{pmatrix},  \,\ab{\Xi}_5 = \begin{pmatrix}37/10 \\ 181/20 \end{pmatrix},\,
\ab{\Xi}_6  = \begin{pmatrix}2 \\ 127/20 \end{pmatrix}, \, \ab{\Xi}_7 = \begin{pmatrix}17/5 \\ 7/2 \end{pmatrix},  \\
\ab{\Xi}_8 & = \begin{pmatrix}29/5 \\ 27/20 \end{pmatrix},  \,\ab{\Xi}_9 = \begin{pmatrix}89/10 \\ 37/20 \end{pmatrix},  \,\ab{\Xi}_{10} = \begin{pmatrix}219/20 \\ 63/20  \end{pmatrix}, \; \ab{\Xi}_{11} := \ab{\Xi}_1. \nonumber 
\end{align}
The third domain is the bilinear-like $G^2$ three-patch domain~$\overline{\Omega}$ presented in Fig.~\ref{fig:domains} (right), where the individual patches are parameterized by bi-cubic geometry mappings  
from the space $\mathcal{S}_{1/4}^{\ab{3}, \ab{2}}([0,1]^2) \times \mathcal{S}_{1/4}^{\ab{3}, \ab{2}}([0,1]^2)$ as 
\begin{equation}  \label{eq:three_patch_domain_bilinearLike}
 \ab{F}^{(i)} (\ab{\xi}) = \sum_{j_1=0}^{6} \sum_{j_2=0}^{6} \ab{c}^{(i)}_{j_1,j_2} N^{\ab{3},\ab{2}}_{j_1,j_2}(\ab{\xi}),
\end{equation}
{with control points~$\ab{c}^{(i)}_{j_1,j_2}$ specified in}
Table~\ref{three_patch_domain_bilinearLikeTable}.
\begin{table}[htb!]
\normalsize
\centering
\setlength{\tabcolsep}{0.30em}
\begin{tabular}{|ccccccc|}
\hline
\multicolumn{7}{|c|}{$\f{c}_{j_1,j_2}^{(1)}$} \\
\hline
& & & & & & \\[-0.35cm]
$(\frac{119}{15} , \frac{14}{5})$ & 
$( \frac{5971}{720} ,\frac{169}{60})$ & 
$( \frac{721}{80} , \frac{57}{20} )$ & 
$( \frac{1211}{120} , \frac{29}{10} )$ & 
$( \frac{2681}{240} , \frac{59}{20} )$ & 
$( \frac{8561}{720} , \frac{179}{60} )$ & 
$( \frac{49}{4} , 3 )$ \\[0.1cm]
$( \frac{707}{90} , \frac{77}{30} )$ & 
$( \frac{71141}{8640} , \frac{1859}{720} )$ & 
$( \frac{2877}{320} , \frac{209}{80} )$ & 
$( \frac{14581}{1440} , \frac{319}{120} )$ & 
$( \frac{32431}{2880} , \frac{649}{240} )$ & 
$( \frac{103831}{8640} , \frac{1969}{720} )$ & 
$( \frac{595}{48} , \frac{11}{4} )$  \\[0.1cm]
$( \frac{77}{10} , \frac{21}{10} )$ & 
$(  \frac{2597}{320} , \frac{169}{80} )$ & 
$(  \frac{2863}{320} , \frac{171}{80} )$ & 
$(  \frac{1631}{160} , \frac{87}{40} )$ & 
$(  \frac{3661}{320} , \frac{177}{80} )$ & 
$(  \frac{3927}{320} , \frac{179}{80} )$ & 
$(  \frac{203}{16} , \frac{9}{4} )$ \\[0.1cm]
$(\frac{112}{15} , \frac{7}{5} )$ & 
$( \frac{11431}{1440} , \frac{169}{120} )$ & 
$( \frac{1421}{160} , \frac{57}{40} )$ & 
$( \frac{259}{25} , \frac{63}{50} )$ & 
$( \frac{294}{25} , \frac{63}{50} )$ & 
$( \frac{322}{25} , \frac{77}{50} )$ & 
$( \frac{336}{25} , \frac{847}{500} )$ \\[0.1cm]
$( \frac{217}{30} , \frac{7}{10} )$ &
$( \frac{22351}{2880} , \frac{169}{240} )$ &
$( \frac{2821}{320} , \frac{57}{80} )$ &
$( \frac{259}{25} , \frac{7}{25} )$ &
$( \frac{609}{50} , \frac{7}{50} )$ &
$( \frac{336}{25} , \frac{7}{10} )$ &
$( \frac{1757}{125} , \frac{91}{100} )$ \\[0.1cm]
$( \frac{637}{90} , \frac{7}{30} )$ &
$( \frac{66031}{8640} , \frac{169}{720} )$ &
$( \frac{2807}{320} , \frac{19}{80} )$ &
$( \frac{21}{2} , -\frac{14}{25} )$ &
$( \frac{609}{50} , -\frac{49}{50} )$ &
$( \frac{336}{25} , -\frac{14}{25} )$ &
$( \frac{1407}{100} , -\frac{1}{10} )$ \\[0.1cm]
$( 7 , 0 )$ &
$( \frac{91}{12} , 0 )$ &
$( \frac{35}{4} , 0 )$ &
$( \frac{21}{2} , -\frac{112}{125} )$ &
$( \frac{49}{4} , -\frac{77}{50} )$ &
$( \frac{161}{12} , -\frac{49}{45} )$ &
$( 14 , -\frac{7}{10})$ \\[0.1cm]
\hline
\hline 
\multicolumn{7}{|c|}{$\f{c}_{j_1,j_2}^{(2)}$} \\
\hline
& & & & & & \\[-0.35cm]
$( \frac{119}{15} , \frac{14}{5} )$ &
$( \frac{70}{9} , \frac{91}{30} )$ &
$( \frac{112}{15} , \frac{7}{2} )$ &
$( 7 ,\frac{21}{5} )$ &
$( \frac{98}{15} , \frac{49}{10} )$ &
$( \frac{56}{9} , \frac{161}{30} )$ &
$( \frac{91}{15} , \frac{28}{5} )$  \\[0.1cm]
$(\frac{5971}{720} , \frac{169}{60} )$ &
$( \frac{70231}{8640} , \frac{2209}{720} )$ &
$( \frac{22463}{2880} , \frac{857}{240} )$ &
$( \frac{1169}{160} , \frac{173}{40} )$ &
$( \frac{19621}{2880} , \frac{1219}{240} )$ &
$( \frac{56021}{8640} , \frac{4019}{720} )$ &
$( \frac{455}{72} , \frac{35}{6} )$  \\[0.1cm]
$(\frac{721}{80} , \frac{57}{20} )$ &
$( \frac{8477}{960} , \frac{251}{80} )$ &
$( \frac{2709}{320} , \frac{297}{80} )$ &
$( \frac{1267}{160} , \frac{183}{40} )$ &
$( \frac{2359}{320} , \frac{87}{16} )$ &
$( \frac{6727}{960} , \frac{481}{80} )$ &
$( \frac{273}{40} , \frac{63}{10} )$  \\[0.1cm]
$(\frac{1211}{120} , \frac{29}{10} )$ &
$( \frac{14231}{1440} , \frac{389}{120} )$ &
$( \frac{4543}{480} , \frac{157}{40} )$ &
$( \frac{441}{50} , \frac{126}{25} )$ &
$( \frac{203}{25} , \frac{154}{25} )$ &
$( \frac{189}{25} , \frac{343}{50} )$ &
$( \frac{3661}{500} , \frac{3661}{500} )$ \\[0.1cm]
$(\frac{2681}{240} , \frac{59}{20} )$ &
$( \frac{31493}{2880} , \frac{803}{240} )$ &
$( \frac{2009}{192} , \frac{331}{80} )$ &
$( \frac{497}{50} , \frac{273}{50} )$ &
$( \frac{231}{25} , 7 )$ &
$( \frac{413}{50} , \frac{196}{25} )$ &
$( \frac{994}{125} , \frac{4137}{500} )$ \\[0.1cm]
$( \frac{8561}{720} , \frac{179}{60} )$ &
$( \frac{100541}{8640} , \frac{2459}{720} )$ &
$( \frac{32053}{2880} , \frac{1027}{240} )$ &
$( \frac{266}{25} , \frac{287}{50} )$ &
$( \frac{511}{50} , \frac{371}{50} )$ &
$( \frac{231}{25} , \frac{413}{50} )$ &
$( \frac{217}{25} , \frac{259}{30} )$ \\[0.1cm]
$( \frac{49}{4} , 3 )$ &
$( \frac{959}{80} , \frac{69}{20} )$ &
$( \frac{917}{80} , \frac{87}{20} )$ &
$( \frac{5593}{500} , \frac{301}{50} )$ &
$( \frac{5341}{500} , \frac{763}{100} )$ &
$( \frac{777}{80} , \frac{101}{12} )$ &
$( \frac{91}{10} , \frac{35}{4} )$ \\[0.1cm]
\hline
\multicolumn{7}{|c|}{$\f{c}_{j_1,j_2}^{(3)}$} \\
\hline
& & & & & & \\[-0.35cm]
$( \frac{119}{15} , \frac{14}{5} )$ &
$( \frac{707}{90} , \frac{77}{30} )$ &
$( \frac{77}{10} , \frac{21}{10} )$ &
$( \frac{112}{15} , \frac{7}{5} )$ &
$( \frac{217}{30} , \frac{7}{10} )$ &
$( \frac{637}{90} , \frac{7}{30} )$ &
$( 7 , 0 )$ \\[0.1cm]
$(\frac{70}{9} , \frac{91}{30} )$ &
$( \frac{3311}{432} , \frac{1001}{360} )$ &
$( \frac{119}{16} , \frac{91}{40} )$ &
$( \frac{511}{72} , \frac{91}{60} )$ &
$( \frac{973}{144} , \frac{91}{120} )$ &
$( \frac{2821}{432} , \frac{91}{360} )$ &
$( \frac{77}{12} , 0  )$ \\[0.1cm]
$( \frac{112}{15} , \frac{7}{2} )$ &
$( \frac{5243}{720} , \frac{77}{24} )$ &
$( \frac{553}{80} , \frac{21}{8} )$ &
$( \frac{763}{120} , \frac{7}{4} )$ &
$( \frac{1393}{240} , \frac{7}{8} )$ &
$( \frac{3913}{720} , \frac{7}{24} )$ &
$( \frac{21}{4} , 0 )$ \\[0.1cm]
$( 7 , \frac{21}{5} )$ &
$( \frac{161}{24} , \frac{77}{20} )$ &
$( \frac{49}{8} , \frac{63}{20} )$ &
$( \frac{133}{25} , \frac{21}{10} )$ &
$( \frac{231}{50} , \frac{28}{25} )$ &
$( \frac{21}{5} , \frac{7}{50} )$ &
$( \frac{413}{100} , -\frac{21}{100} )$ \\[0.1cm]
$( \frac{98}{15} , \frac{49}{10} )$ &
$( \frac{4417}{720} , \frac{539}{120} )$ &
$( \frac{427}{80} , \frac{147}{40} )$ &
$( \frac{21}{5} , \frac{14}{5} )$ &
$( \frac{84}{25} , \frac{91}{50} )$ &
$( \frac{77}{25} , \frac{14}{25} )$ &
$( \frac{301}{100} , -\frac{7}{250} )$ \\[0.1cm]
$( \frac{56}{9} , \frac{161}{30} )$ &
$( \frac{2485}{432} , \frac{1771}{360} )$ &
$( \frac{77}{16} , \frac{161}{40} )$ &
$( \frac{7}{2} , \frac{7}{2} )$ &
$( \frac{119}{50} , \frac{147}{50} )$ &
$( \frac{21}{10} , \frac{77}{50} )$ &
$( \frac{119}{60} , \frac{21}{25} )$ \\[0.1cm]
$(\frac{91}{15} , \frac{28}{5} )$ &
$( \frac{1001}{180} , \frac{77}{15} )$ &
$( \frac{91}{20} , \frac{21}{5} )$ &
$( \frac{392}{125} , \frac{987}{250} )$ &
$( \frac{987}{500} , \frac{1729}{500} )$ &
$( \frac{14}{9} , \frac{133}{60} )$ &
$( \frac{7}{5} , \frac{7}{5}  )$ \\[0.1cm]
\hline
\end{tabular}
\caption{Control points~$\ab{c}_{j_1,j_2}^{(i)}$, $i=1,2,3$, of the bicubic bilinear-like $G^2$ three-patch spline geometry \eqref{eq:three_patch_domain_bilinearLike} shown in Fig.~\ref{fig:domains} (right).}
\label{three_patch_domain_bilinearLikeTable} 
\end{table}

\begin{figure}[htb!]
    \centering
    \includegraphics[scale=0.22]{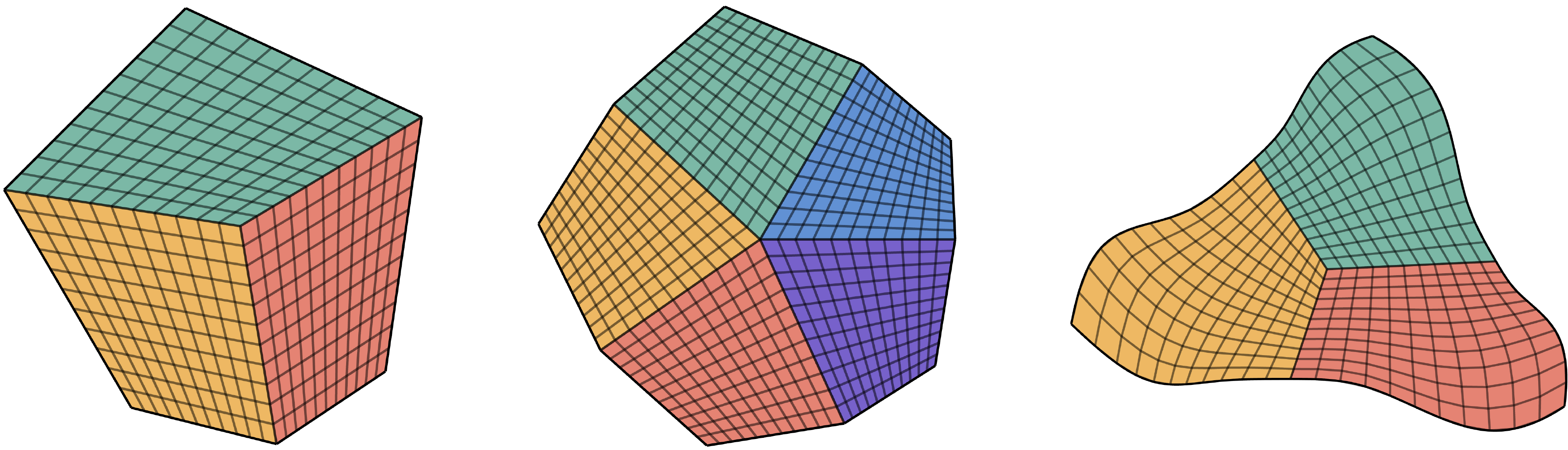}
\caption{Plots of the bilinear three-patch domain \eqref{eq:verticesThreePatch}, left, of the bilinear five-patch domain \eqref{eq:verticesFivePatch}, middle, and of the bilinear-like $G^2$ three-patch domain \eqref{eq:three_patch_domain_bilinearLike}, right.}
    \label{fig:domains}
\end{figure}

While in the first two examples we use $C^s$-smooth discretization spaces~$\mathcal{W}^s_h$ based on the mixed degree underlying spline spaces $\mathcal{S}_h^{(\ab{\sm}+\ab{1}, \ab{2\sm}+\ab{1}),(\ab{\sm},\ab{\sm})}([0,1]^2)$, $\sm=1,2$ (i.e. Case A), in the remaining two examples $C^s$-smooth discretization spaces~$\mathcal{W}^s_h$ based on the mixed regularity underlying spaces $\mathcal{S}_h^{(\ab{2\sm}+\ab{1}, \ab{2\sm}+\ab{1}),(\ab{2\sm},\ab{\sm})}([0,1]^2)$, $\sm=1,2$ (i.e. Case B) will be considered. 
Note that in all examples below, the number of degrees of freedom involved is much lower in comparison to using $C^s$-smooth discretization spaces~$\mathcal{W}^s_h$ based on the usual underlying spline spaces $\mathcal{S}_h^{(\ab{2\sm}+\ab{1}, \ab{2\sm}+\ab{1}),(\ab{\sm},\ab{\sm})}([0,1]^2)$, $s=1,2$, with the same degree and regularity everywhere. 

\begin{ex} \label{ex:ThreeFivepatchdomainBiharmonic}
We solve the biharmonic equation \eqref{eq:biharmonic_problem_Galerkin} 
over the three multi-patch domains presented in Fig.~\ref{fig:domains}. For this, we use $C^1$-smooth discretization spaces~$\mathcal{W}^1_h$ 
which are based on the underlying mixed degree spline spaces $\mathcal{S}_h^{(\ab{2}, \ab{3}),(\ab{1},\ab{1})}([0,1]^2)$
with mesh sizes $h= h_{0}/2^j$, $j=0,\ldots,3$, and $h_{0}=1/6$. 
Since the parameterizations $\ab{F}^{(i)}$ of the individual patches $\Omega^{(i)}$ should fulfill the relation
$
\ab{F}^{(i)} \in \mathcal{S}_{h_0}^{\ab{2}, \ab{1}}([0,1]^2) \times \mathcal{S}_{h_0}^{\ab{2}, \ab{1}}([0,1]^2),
$
we have to slightly modify the bilinear-like $G^2$ three-patch domain \eqref{eq:three_patch_domain_bilinearLike} in such a way that the individual patches are described by bi-quadratic geometry mappings  
from the space $\mathcal{S}_{1/6}^{\ab{2}, \ab{1}}([0,1]^2) \times \mathcal{S}_{1/6}^{\ab{2}, \ab{1}}([0,1]^2)$, which results technically in just a bilinear-like $G^1$ three-patch domain, see Fig.~\ref{fig:galerkin_biharmonic} (top, right), 
but nevertheless sufficient for solving the biharmonic equation with $C^1$-smooth functions.
We study the computed relative errors~\eqref{eq:eqiuv2seminorms} considering the $L^2$, $H^1$ and $H^2$-(semi)norms, and observe
for all three domains a convergence of order~$\mathcal{O}(h^{p_1}) = \mathcal{O}(h^{2}) $ for the $L^2$ and $H^1$-(semi)norms, 
{while studying the $H^2$-seminorm an order of $\mathcal{O}(h^{p_1-1}) = \mathcal{O}(h^{1})$ is obtained}, cf. Fig.~\ref{fig:galerkin_biharmonic} (below).
\begin{figure}[htb!]
    \centering
    \includegraphics[scale=0.218]{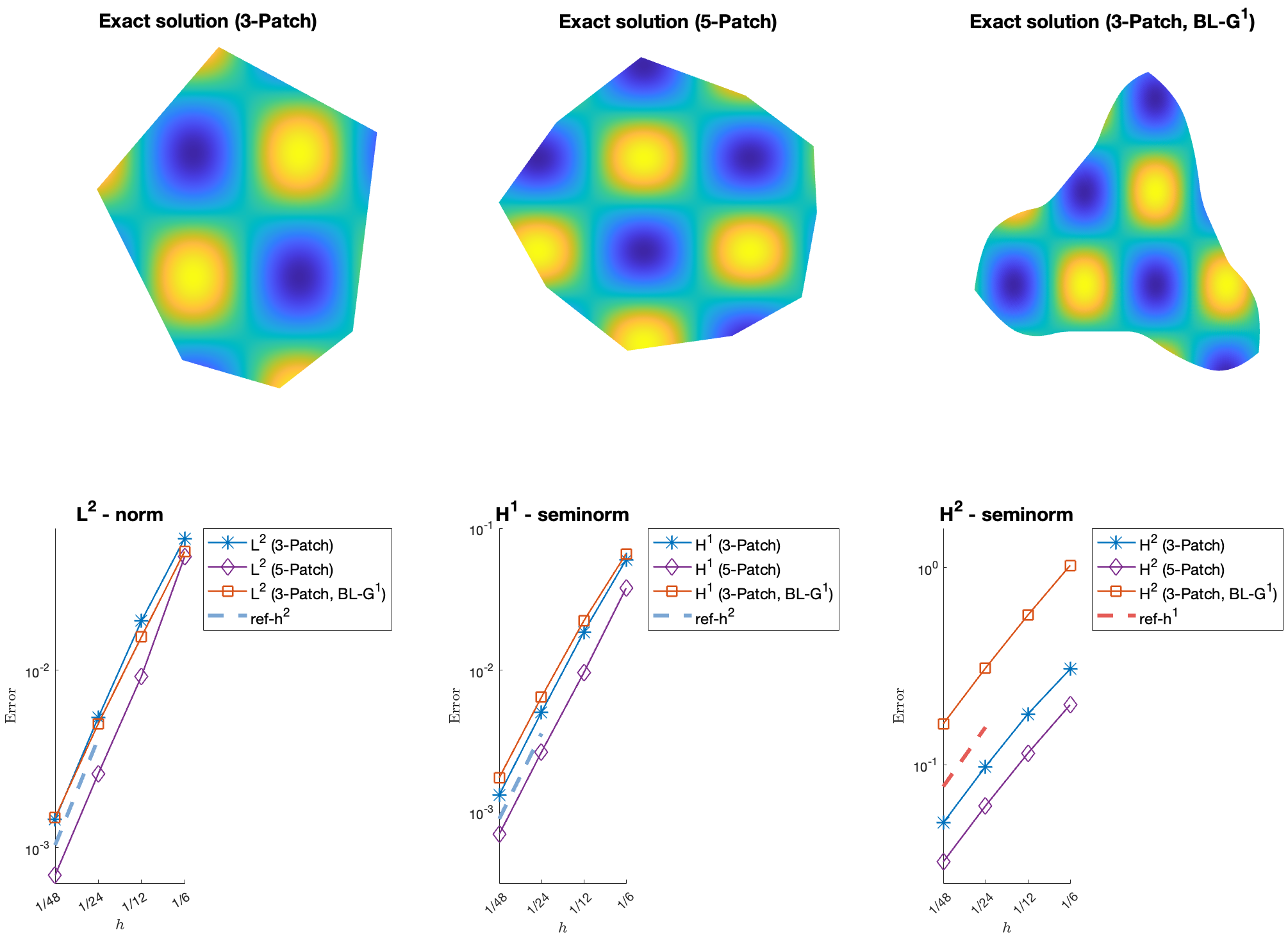}

\caption{Example~\ref{ex:ThreeFivepatchdomainBiharmonic}. Plots of the considered exact solution~\eqref{eq:exactSolution} over the bilinear three-patch domain (top, left), the bilinear five-patch domain (top, middle) and the bilinear-like $G^1$ three-patch domain (top, right), as well as plots of the relative errors~\eqref{eq:eqiuv2seminorms} computed with respect to the $L^2$-norm and with respect to the $H^1$ and $H^2$-seminorm (below, left to right) for solving the biharmonic equation~\eqref{eq:biharmonic_problem_Galerkin} using the $C^1$-smooth discretization spaces~$\mathcal{W}^1_h$ based on the underlying mixed degree spline space $\mathcal{S}_h^{(\ab{2}, \ab{3}),(\ab{1},\ab{1})}([0,1]^2)$.}
    \label{fig:galerkin_biharmonic}
\end{figure}
\end{ex}

\begin{ex} \label{ex:ThreeFivepatchdomainTriharmonic}
In this example, we solve the triharmonic equation \eqref{eq:triharmonic_problem} 
over two of the multi-patch domains given in Fig.~\ref{fig:domains}, by choosing $C^2$-smooth discretization spaces~$\mathcal{W}^2_h$  which are based on the underlying mixed degree spline spaces $\mathcal{S}_h^{(\ab{3}, \ab{5}),(\ab{2},\ab{2})}([0,1]^2)$
with mesh sizes $h= h_{0}/2^j$, $j=0,\ldots,3$, and $h_{0}=1/5$.  Thereby, the first considered domain is the  bilinear five-patch domain \eqref{eq:verticesFivePatch}  
given in Fig.~\ref{fig:domains} (middle), while the second one is again an approximation of the bilinear-like $G^2$ three-patch domain \eqref{eq:three_patch_domain_bilinearLike}, now such that 
$
\ab{F}^{(i)} \in \mathcal{S}_{h_0}^{\ab{3}, \ab{2}}([0,1]^2) \times \mathcal{S}_{h_0}^{\ab{3}, \ab{2}}([0,1]^2),
$
cf.~Fig.~\ref{fig:galerkin_triharmonic} (top, middle).
By comparing the resulting relative errors~\eqref{eq:eqiuv2seminorms},  the numerical results show for both domains convergence rates of order $\mathcal{O}(h^{p_1-1}) = \mathcal{O}(h^{2}) $ for the $L^2$, $H^1$ and $H^2$-(semi)norm and of order $\mathcal{O}(h^{p_1-2}) = \mathcal{O}(h^{1})$ for the $H^3$-seminorm, cf. Fig.~\ref{fig:galerkin_triharmonic}. 
\begin{figure}[htb!]
    \centering
    \includegraphics[scale=0.218]{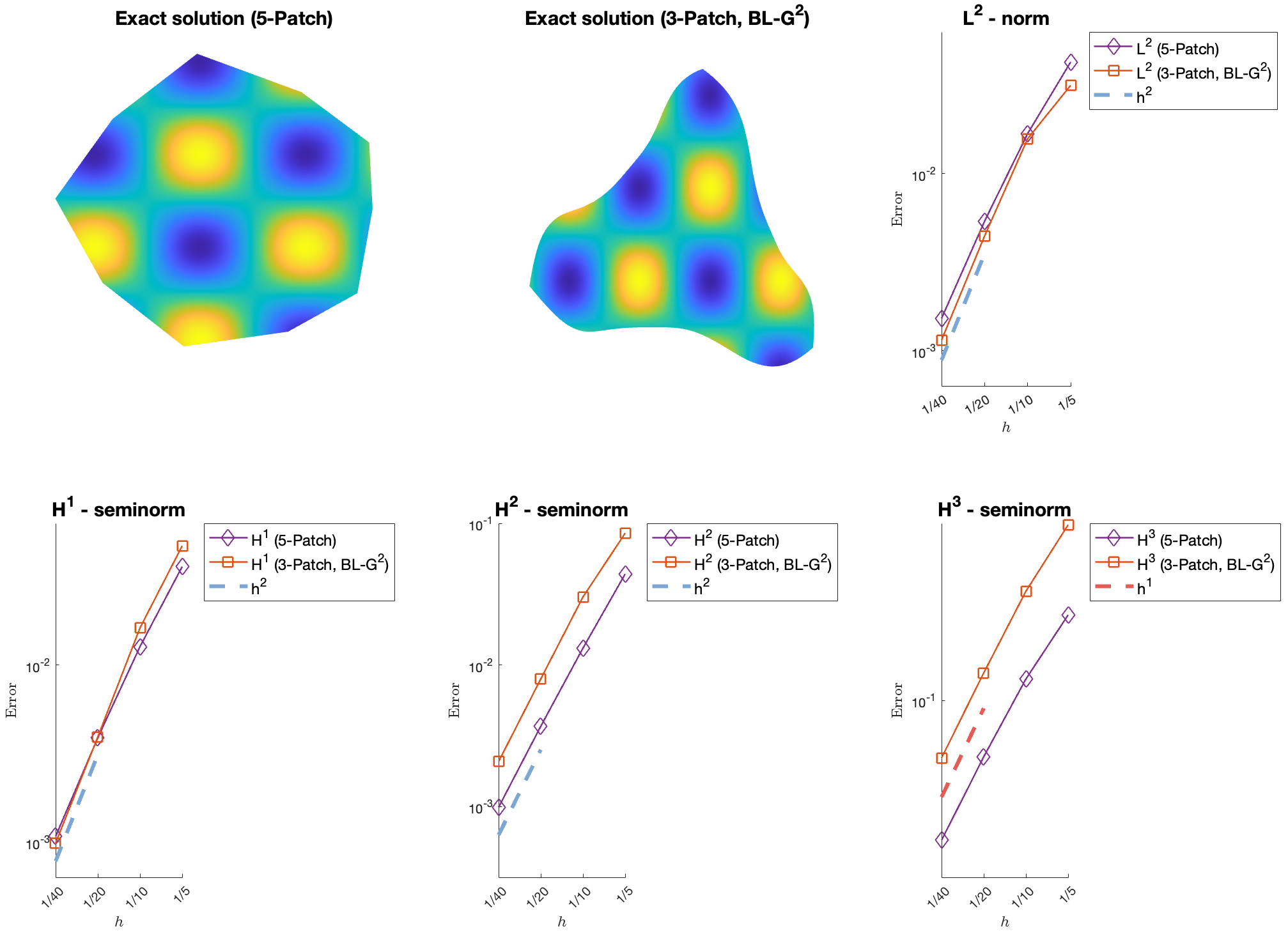}
\caption{Example~\ref{ex:ThreeFivepatchdomainTriharmonic}. Plots of the considered exact solution~\eqref{eq:exactSolution} over the bilinear five-patch domain (top, left) and of the bilinear-like $G^2$ three-patch domain (top, middle), as well as plots of the relative errors~\eqref{eq:eqiuv2seminorms} computed with respect to the $L^2$-norm (top, right) and with respect to the $H^1$,$H^2$ and $H^3$-seminorm (below, left to right) for solving the triharmonic equation~\eqref{eq:triharmonic_problem} using the $C^2$-smooth discretization spaces~$\mathcal{W}^2_h$ based on the underlying mixed degree spline space $\mathcal{S}_h^{(\ab{3}, \ab{5}),(\ab{2},\ab{2})}([0,1]^2)$.}
    \label{fig:galerkin_triharmonic}
\end{figure} 
\end{ex}

Note that for both cases, i.e. for solving the biharmonic and the triharmonic equation by employing $C^s$-smooth mixed degree isogeometric spline spaces with $p_1=\sm+1$ and $\sm=1,2$, the obtained orders are the expected ones for isogeometric  spline spaces of degree $p_1=\sm+1$, $\sm=1,2$. Note also that some of the orders could be increased by using the weak Galerkin approach instead of the standard one, cf. \cite{WangYe2013, ZhangJun2015, ZhangYeLin2019}.

\begin{ex} \label{ex:ThreeFivepatchdomainBiharmonicCase2}
Let us again solve the biharmonic equation \eqref{eq:biharmonic_problem_Galerkin} 
over the three multi-patch domains presented in Fig.~\ref{fig:domains}.   
Now, the underlying spline spaces for the $C^1$-smooth discretization spaces~$\mathcal{W}^1_h$ are the mixed regularity spline spaces $\mathcal{S}_h^{(\ab{3}, \ab{3}),(\ab{2},\ab{1})}([0,1]^2)$ with mesh sizes $h= h_{0}/2^j$, $j=0,\ldots,3$, and $h_{0}=1/4$. 
By studying the numerical results, we observe  for all three domains convergence rates of optimal order $\mathcal{O}(h^{p_1+1}) = \mathcal{O}(h^{4}) $, $\mathcal{O}(h^{p_1}) = \mathcal{O}(h^{3}) $ and $\mathcal{O}(h^{p_1-1}) = \mathcal{O}(h^{2}) $ for the $L^2$, $H^1$ and $H^2$-(semi)norms, respectively. cf. Fig.~\ref{fig:galerkin_biharmonic_Case2} (below). 
\begin{figure}[htb!]
    \centering
    \includegraphics[scale=0.218]{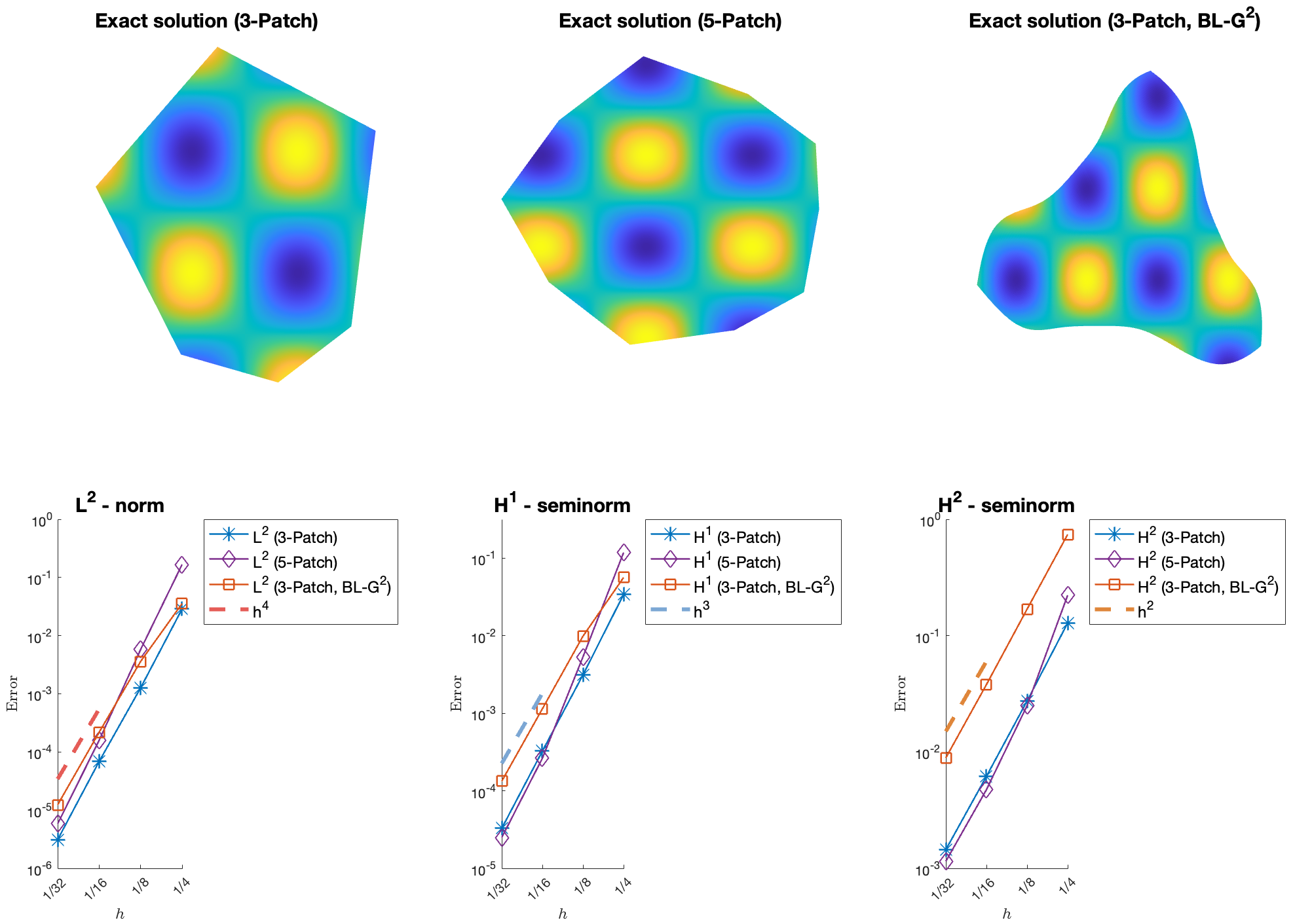}\caption{Example~\ref{ex:ThreeFivepatchdomainBiharmonicCase2}.
    Plots of the bilinear three-patch domain (top, left), of the bilinear five-patch domain (top, middle) and of the bilinear-like $G^2$ three-patch domain (top, right) with the considered exact solution~\eqref{eq:exactSolution}, as well as plots of the relative errors~\eqref{eq:eqiuv2seminorms} computed with respect to the $L^2$-norm and with respect to the $H^1$ and $H^2$-seminorm (below, left to right) for solving the biharmonic equation~\eqref{eq:biharmonic_problem_Galerkin} using the $C^1$-smooth discretization spaces~$\mathcal{W}^1_h$ based on the underlying mixed regularity spline space $\mathcal{S}_h^{(\ab{3}, \ab{3}),(\ab{2},\ab{1})}([0,1]^2)$. }
  \label{fig:galerkin_biharmonic_Case2}
\end{figure} 
\end{ex}

\begin{ex} \label{ex:ThreeFivepatchdomainTriharmonicCase2}
In the last example, we solve again the triharmonic equation \eqref{eq:triharmonic_problem}, now by choosing $C^2$-smooth discretization spaces~$\mathcal{W}^2_h$ based on underlying 
mixed regularity spline spaces $\mathcal{S}_h^{(\ab{5}, \ab{5}),(\ab{4},\ab{2})}([0,1]^2)$  with mesh sizes $h= h_{0}/2^j$, $j=0,\ldots,3$, and $h_{0}=1/5$. We consider on the one hand the bilinear five-patch domain \eqref{eq:verticesFivePatch} presented in Fig.~\ref{fig:domains}, and on the other hand the bilinear-like $G^2$ three-patch domain, which is an approximation of the domain \eqref{eq:three_patch_domain_bilinearLike}
in such a way that the individual patches are described by bi-quintic geometry mappings 
from the space $\mathcal{S}_{h_0}^{\ab{5}, \ab{4}}([0,1]^2) \times \mathcal{S}_{h_0}^{\ab{5}, \ab{4}}([0,1]^2)$, see Fig.~\ref{fig:galerkin_triharmonic_Case2} (top, middle).
The numerical results indicate for both domains convergence rates of optimal order $\mathcal{O}(h^{p_1+1}) = \mathcal{O}(h^{6}) $, $\mathcal{O}(h^{p_1}) = \mathcal{O}(h^{5}) $, $\mathcal{O}(h^{p_1-1}) = \mathcal{O}(h^{4}) $ and $\mathcal{O}(h^{p_1-2}) = \mathcal{O}(h^{3}) $ for the $L^2$, $H^1$, $H^2$ and $H^3$-(semi)norm, respectively. cf. Fig.~\ref{fig:galerkin_triharmonic_Case2}. 
\begin{figure}[htb!]
    \centering
    \includegraphics[scale=0.218]{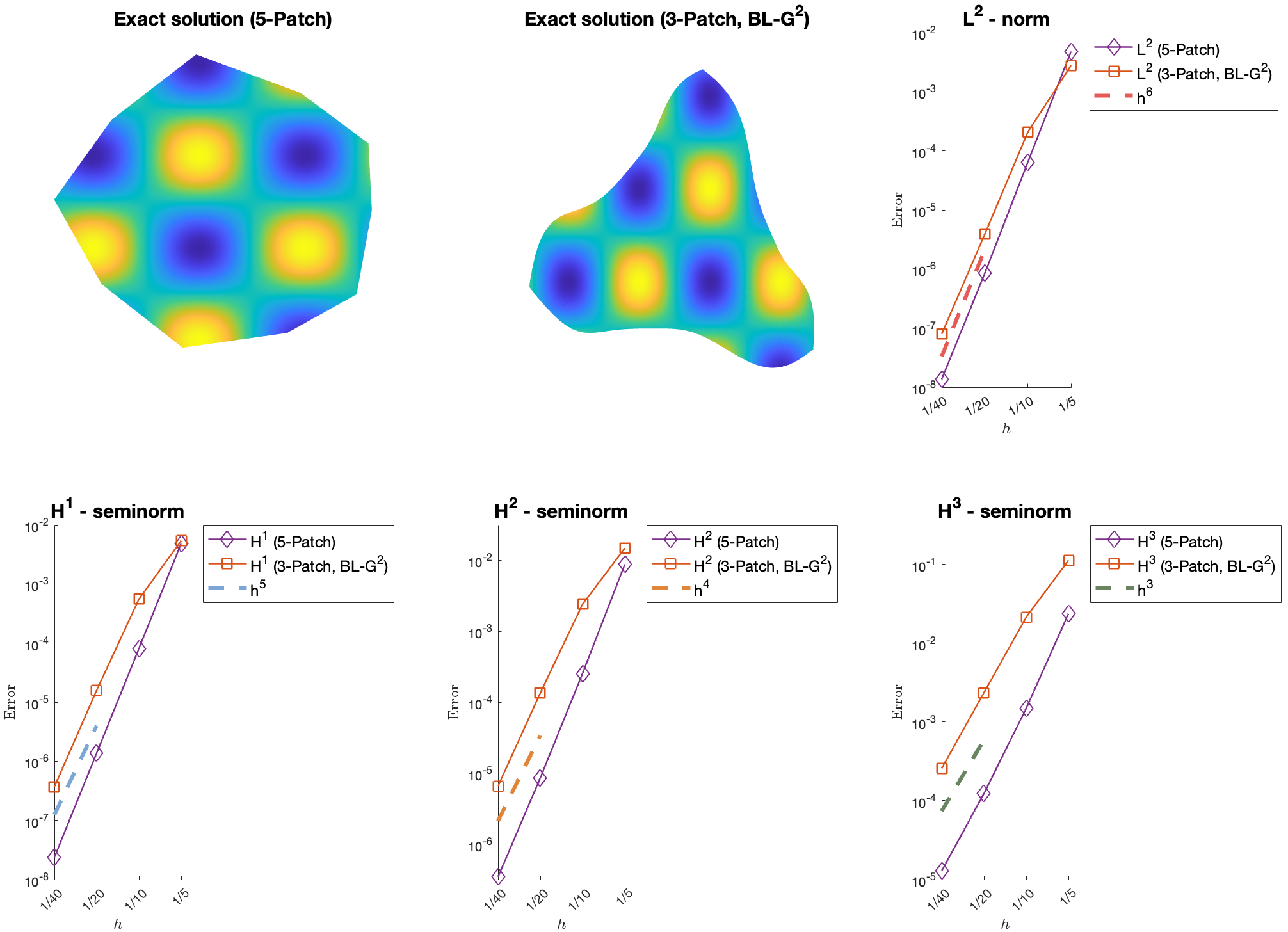}\caption{Example~\ref{ex:ThreeFivepatchdomainTriharmonicCase2}. Plots of the bilinear five-patch domain (top, left) and of the bilinear-like $G^2$ three-patch domain (top, middle) with the considered exact solution~\eqref{eq:exactSolution}, as well as plots of the relative errors~\eqref{eq:eqiuv2seminorms} computed
    with respect to the $L^2$-norm (top, right) and with respect to the $H^1$,$H^2$ and $H^3$-seminorm (below, left to right) for solving the triharmonic equation~\eqref{eq:triharmonic_problem} using the $C^2$-smooth discretization spaces~$\mathcal{W}^2_h$ based on the underlying mixed regularity spline space $\mathcal{S}_h^{(\ab{5}, \ab{5}),(\ab{4},\ab{2})}([0,1]^2)$. }
\label{fig:galerkin_triharmonic_Case2}
\end{figure} 
\end{ex}

\section{Conclusion} \label{sec:Conclusion}

We presented the 
design
of $C^s$-smooth mixed degree and regularity isogeometric spline spaces over planar bilinearly parameterized multi-patch geometries, 
and expanded the construction, providing an example, to the richer class of bilinear-like $G^s$ multi-patch domains, allowing the modelling of complex geometries containing curved edges.
The generated $C^s$-smooth mixed degree and regularity spline spaces possess the spline degree $p=2s+1$ and regularity $r=\sm$ in a small neighborhood around the edges and vertices of the multi-patch domain, and can have a reduced degree and/or an increased regularity in the interior of the patches. The design of these $C^s$-smooth spline spaces is based on the use of a novel underlying mixed degree and regularity spline space defined on the unit square~$[0,1]^2$. Some properties of  the introduced underlying mixed degree and regularity spline space are also presented.  

We studied several numerical examples for the two most practical and interesting cases of the $C^s$-smooth mixed degree and regularity isogeometric spline space, namely on the one hand the space with the same regularity everywhere, but with the smallest possible spline degree in the interior of the patches (called Case A), and on the other hand the space with the same spline degree  everywhere, but with the largest possible spline regularity for the interior of the patches (called Case B).
The biharmonic and the triharmonic equations are solved using the isogeometric Galerkin approach and the convergence of the results regarding the $L^2$-norm as well as (equivalents of) the $H^m$-seminorms, $1\leq m\leq \sm$, is studied.


A first research goal in future 
could be the application of $C^s$-smooth mixed degree and regularity isogeometric spline spaces to other methods like to isogeometric collocation. 
Furthermore, an extension to other problems such as Kirchhoff plate or Kirchhoff-Love shell problem as well as expanding the presented design to multi-patch surfaces and multi-patch volumes could also be an 
interesting research topic.

\paragraph*{\bf Acknowledgment}

M. Kapl has been partially supported by the Austrian Science Fund (FWF) through the project P~33023-N. V.~Vitrih has been partially supported by the 
Slovenian Research and Innovation Agency (research program P1-0404 and research projects N1-0296 and J1-4414). A.~Kosma\v c has been partially supported by the Slovenian Research and Innovation Agency (research program P1-0404, research project N1-0296 and Young Researchers Grant). This support is gratefully acknowledged.



\begin{thebibliography}{10}

\bibitem{ArReKlSi23}
J.~Arf, M.~Reichle, S.~Klinkel, and B.~Simeon.
\newblock Scaled boundary isogeometric analysis with ${C}^1$ coupling for
  {K}irchhoff plate theory.
\newblock {\em Comput. Methods Appl. Mech. Engrg.}, 415:116198, 2023.

\bibitem{BaDe15}
A.~Bartezzaghi, L.~Ded\`{e}, and A.~Quarteroni.
\newblock Isogeometric analysis of high order partial differential equations on
  surfaces.
\newblock {\em Comput. Methods Appl. Mech. Engrg.}, 295:446 -- 469, 2015.

\bibitem{BeCaMo2017}
C.~V. Beccari, G.~Casciola, and S.~Morigi.
\newblock On multi-degree splines.
\newblock {\em Computer Aided Geometric Design}, 58:8--23, 2017.

\bibitem{BeMa14}
M.~Bercovier and T.~Matskewich.
\newblock {\em Smooth {B}\'{e}zier Surfaces over Unstructured Quadrilateral
  Meshes}.
\newblock Lecture Notes of the Unione Matematica Italiana, Springer, 2017.

\bibitem{BlMoVi17}
A.~Blidia, B.~Mourrain, and N.~Villamizar.
\newblock {G}$^1$-smooth splines on quad meshes with 4-split macro-patch
  elements.
\newblock {\em Comput. Aided Geom. Des.}, 52–-53:106 -- 125, 2017.

\bibitem{BlMoXu20}
A.~Blidia, B.~Mourrain, and G.~Xu.
\newblock Geometrically smooth spline bases for data fitting and simulation.
\newblock {\em Comput. Aided Geom. Des.}, 78:101814, 2020.

\bibitem{Bouclier2017}
R.~Bouclier, J.~C. Passieux, and M.~Sala{\"{u}}n.
\newblock {Development of a new, more regular, mortar method for the coupling
  of NURBS subdomains within a NURBS patch: Application to a non-intrusive
  local enrichment of NURBS patches}.
\newblock {\em Comput. Methods Appl. Mech. Engrg.}, 316:123--150, 2017.

\bibitem{ChAnRa18}
C.L. Chan, C.~Anitescu, and T.~Rabczuk.
\newblock Isogeometric analysis with strong multipatch {C}$^1$-coupling.
\newblock {\em Comput. Aided Geom. Design}, 62:294--310, 2018.

\bibitem{ChAnRa19}
C.L. Chan, C.~Anitescu, and T.~Rabczuk.
\newblock Strong multipatch {C}$^1$-coupling for isogeometric analysis on {2D}
  and {3D} domains.
\newblock {\em Comput. Methods Appl. Mech. Engrg.}, 357:112599, 2019.

\bibitem{CoSaTa16}
A.~Collin, G.~Sangalli, and T.~Takacs.
\newblock Analysis-suitable {G}$^1$ multi-patch parametrizations for {C}$^1$
  isogeometric spaces.
\newblock {\em Comput. Aided Geom. Des.}, 47:93 -- 113, 2016.

\bibitem{CottrellBook}
J.~A. Cottrell, T.~J.~R. Hughes, and Y.~Bazilevs.
\newblock {\em Isogeometric Analysis: Toward Integration of {CAD} and {FEA}}.
\newblock John Wiley \& Sons, Chichester, England, 2009.

\bibitem{DiSchWoHe19}
M.~Dittmann, S.~Schu{\ss}, B.~Wohlmuth, and C.~Hesch.
\newblock Weak ${C}^n$ coupling for multipatch isogeometric analysis in solid
  mechanics.
\newblock {\em Int. J. Numer. Methods Eng.}, 118(11):678--699, 2019.

\bibitem{DiSchWoHe20}
M~Dittmann, S.~Schu{\ss}, B.~Wohlmuth, and C.~Hesch.
\newblock Crosspoint modification for multi-patch isogeometric analysis.
\newblock {\em Comput. Methods Appl. Mech. Engrg.}, 360:112768, 2020.

\bibitem{Dornisch2015}
W.~Dornisch, G.~Vitucci, and S.~Klinkel.
\newblock {The weak substitution method - an application of the mortar method
  for patch coupling in NURBS-based isogeometric analysis}.
\newblock {\em Int. J. Numer. Methods Eng.}, 103(3):205--234, 2015.

\bibitem{FaJuKaTa22}
A.~Farahat, B.~J\"uttler, M.~Kapl, and T.~Takacs.
\newblock Isogeometric analysis with ${C}^1$-smooth functions over multi-patch
  surfaces.
\newblock {\em Comput. Methods Appl. Mech. Engrg.}, 403:115706, 2023.

\bibitem{FaVeKiKa23}
A.~Farahat, H.~M. Verhelst, J.~Kiendl, and M.~Kapl.
\newblock Isogeometric analysis for multi-patch structured {K}irchhoff–{L}ove
  shells.
\newblock {\em Comput. Methods Appl. Mech. Engrg.}, 411:116060, 2023.

\bibitem{Fa97}
G.~Farin.
\newblock {\em Curves and Surfaces for Computer-Aided Geometric Design}.
\newblock Academic Press, 1997.

\bibitem{GiJuSp2012}
C.~Giannelli, B.~Jüttler, and H.~Speleers.
\newblock {THB}-splines: The truncated basis for hierarchical splines.
\newblock {\em Computer Aided Geometric Design}, 29(7):485--498, 2012.
\newblock Geometric Modeling and Processing 2012.

\bibitem{Guo2015881}
Y.~Guo and M.~Ruess.
\newblock Nitsche's method for a coupling of isogeometric thin shells and
  blended shell structures.
\newblock {\em Comp. Methods Appl. Mech. Engrg.}, 284:881--905, 2015.

\bibitem{HeJoPrWuKiHs2019}
A.~J. Herrema, E.~L. Johnson, D.~Proserpio, M.~C.~H. Wu, J.~Kiendl, and M.-C.
  Hsu.
\newblock {Penalty coupling of non-matching isogeometric {K}irchhoff--{L}ove
  shell patches with application to composite wind turbine blades}.
\newblock {\em Comput. Methods Appl. Mech. Engrg.}, 346:810--840, 2019.

\bibitem{Horger2019}
T.~Horger, A.~Reali, B.~Wohlmuth, and L.~Wunderlich.
\newblock {A hybrid isogeometric approach on multi-patches with applications to
  Kirchhoff plates and eigenvalue problems}.
\newblock {\em Comput. Methods Appl. Mech. Engrg.}, 348:396--408, 2019.

\bibitem{HoLa93}
J.~Hoschek and D.~Lasser.
\newblock {\em Fundamentals of computer aided geometric design}.
\newblock A K Peters Ltd., Wellesley, MA, 1993.

\bibitem{Hu2018}
Q.~Hu, F.~Chouly, P.~Hu, G.~Cheng, and S.~P.~A. Bordas.
\newblock {Skew-symmetric Nitsche's formulation in isogeometric analysis:
  Dirichlet and symmetry conditions, patch coupling and frictionless contact}.
\newblock {\em Comput. Methods Appl. Mech. Engrg.}, 341:188--220, nov 2018.

\bibitem{HuCoBa04}
T.~J.~R. Hughes, J.~A. Cottrell, and Y.~Bazilevs.
\newblock Isogeometric analysis: {CAD}, finite elements, {NURBS}, exact
  geometry and mesh refinement.
\newblock {\em Comput. Methods Appl. Mech. Engrg.}, 194(39-41):4135--4195,
  2005.

\bibitem{HuSaTaTo21}
T.~J.~R. Hughes, G.~Sangalli, T.~Takacs, and D.~Toshniwal.
\newblock Chapter 8 - smooth multi-patch discretizations in isogeometric
  analysis.
\newblock In Andrea Bonito and Ricardo~H. Nochetto, editors, {\em Geometric
  Partial Differential Equations - Part II}, volume~22 of {\em Handbook of
  Numerical Analysis}, pages 467--543. Elsevier, 2021.

\bibitem{ZhangJun2015}
H.~Jun and S.~Zhang.
\newblock The minimal conforming ${H}^k$ finite element spaces on ${R}^n$
  rectangular grids.
\newblock {\em Mathematics of Computation}, 84(292):563--579, 2015.

\bibitem{KaBuBeJu16}
M.~Kapl, F.~Buchegger, M.~Bercovier, and B.~J\"uttler.
\newblock Isogeometric analysis with geometrically continuous functions on
  planar multi-patch geometries.
\newblock {\em Comput. Methods Appl. Mech. Engrg.}, 316:209 -- 234, 2017.

\bibitem{KaKoVi24}
M.~Kapl, A.~Kosma\v{c}, and V.~Vitrih.
\newblock Isogeometric collocation for solving the biharmonic equation over
  planar multi-patch domains.
\newblock {\em Computer Methods in Applied Mechanics and Engineering},
  424:116882, 2024.

\bibitem{KaSaTa19b}
M.~Kapl, G.~Sangalli, and T.~Takacs.
\newblock Isogeometric analysis with {C}$^{1}$ functions on unstructured
  quadrilateral meshes.
\newblock {\em The SMAI journal of computational mathematics}, 5:67--86, 2019.

\bibitem{KaSaTa21}
M.~Kapl, G.~Sangalli, and T.~Takacs.
\newblock A family of ${C}^1$ quadrilateral finite elements.
\newblock {\em Adv. Comp. Math.}, 47(6):82, 2021.

\bibitem{KaVi17b}
M.~Kapl and V.~Vitrih.
\newblock Space of {C}$^2$-smooth geometrically continuous isogeometric
  functions on planar multi-patch geometries: {D}imension and numerical
  experiments.
\newblock {\em Comput. Math. Appl.}, 73(10):2319--2338, 2017.

\bibitem{KaVi17a}
M.~Kapl and V.~Vitrih.
\newblock Space of {C}$^{2}$-smooth geometrically continuous isogeometric
  functions on two-patch geometries.
\newblock {\em Comput. Math. Appl.}, 73(1):37--59, 2017.

\bibitem{KaVi17c}
M.~Kapl and V.~Vitrih.
\newblock Dimension and basis construction for ${C}^{2}$-smooth isogeometric
  spline spaces over bilinear-like ${G}^{2}$ two-patch parameterizations.
\newblock {\em J. Comput. Appl. Math.}, 335:289--311, 2018.

\bibitem{KaVi19a}
M.~Kapl and V.~Vitrih.
\newblock Solving the triharmonic equation over multi-patch planar domains
  using isogeometric analysis.
\newblock {\em J. Comput. Appl. Math.}, 358:385--404, 2019.

\bibitem{KaVi20b}
M.~Kapl and V.~Vitrih.
\newblock ${C}^s$-smooth isogeometric spline spaces over planar multi-patch
  parameterizations.
\newblock {\em Advances in Computational Mathematics}, 47:47, 2021.

\bibitem{Pe15-2}
K.~Kar{\v c}iauskas, T.~Nguyen, and J.~Peters.
\newblock {G}eneralizing bicubic splines for modeling and {IGA} with irregular
  layout.
\newblock {\em Comput.-Aided Des.}, 70:23--35, 2016.

\bibitem{KaPe17}
K.~Kar{\v c}iauskas and J.~Peters.
\newblock Refinable ${G}^1$ functions on ${G}^1$ free-form surfaces.
\newblock {\em Comput. Aided Geom. Des.}, 54:61--73, 2017.

\bibitem{KaPe18}
K.~Kar{\v c}iauskas and J.~Peters.
\newblock Refinable bi-quartics for design and analysis.
\newblock {\em Comput.-Aided Des.}, pages 204--214, 2018.

\bibitem{kiendl-bazilevs-hsu-wuechner-bletzinger-10}
J.~Kiendl, Y.~Bazilevs, M.-C. Hsu, R.~W{\"u}chner, and K.-U. Bletzinger.
\newblock The bending strip method for isogeometric analysis of
  {K}irchhoff-{L}ove shell structures comprised of multiple patches.
\newblock {\em Comput. Methods Appl. Mech. Engrg.}, 199(35):2403--2416, 2010.

\bibitem{LaSch07}
M.-J. Lai and L.~L. Schumaker.
\newblock {\em Spline functions on triangulations}, volume 110 of {\em
  Encyclopedia of Mathematics and its Applications}.
\newblock Cambridge University Press, Cambridge, 2007.

\bibitem{Leonetti2020}
L.~Leonetti, F.~S. Liguori, D.~Magisano, J.~Kiendl, A.~Reali, and G.~Garcea.
\newblock {A robust penalty coupling of non-matching isogeometric
  {K}irchhoff--{L}ove shell patches in large deformations}.
\newblock {\em Comput. Methods Appl. Mech. Engrg.}, 371, 2020.

\bibitem{mourrain2015geometrically}
B.~Mourrain, R.~Vidunas, and N.~Villamizar.
\newblock Dimension and bases for geometrically continuous splines on surfaces
  of arbitrary topology.
\newblock {\em Comput. Aided Geom. Des.}, 45:108 -- 133, 2016.

\bibitem{ZhangYeLin2019}
L.~Mu, X.~Ye, and S.~Zhang.
\newblock Development of a ${P}_2$ element with optimal ${L}^2$ convergence for
  biharmonic equation.
\newblock {\em Numerical Methods for Partial Differential Equations},
  35(4):1497--1508, 2019.

\bibitem{Peters2}
T.~Nguyen, K.~Kar{\v c}iauskas, and J.~Peters.
\newblock A comparative study of several classical, discrete differential and
  isogeometric methods for solving {P}oisson's equation on the disk.
\newblock {\em Axioms}, 3(2):280--299, 2014.

\bibitem{NgKaPe15}
T.~Nguyen, K.~Kar{\v c}iauskas, and J.~Peters.
\newblock ${C}^{1}$ finite elements on non-tensor-product 2d and 3d manifolds.
\newblock {\em Applied Mathematics and Computation}, 272:148--158, 2016.

\bibitem{NgPe16}
T.~Nguyen and J.~Peters.
\newblock Refinable ${C}^{1}$ spline elements for irregular quad layout.
\newblock {\em Comput. Aided Geom. Des.}, 43:123--130, 2016.

\bibitem{Nguyen2014}
V.~P. Nguyen, P.~Kerfriden, M.~Brino, S.~P.A. Bordas, and E.~Bonisoli.
\newblock {Nitsche's method for two and three dimensional NURBS patch
  coupling}.
\newblock {\em Computational Mechanics}, 53(6):1163--1182, 2014.

\bibitem{ReArSiKl23}
M.~F.~M. Reichle, J.~Arf, B.~Simeon, and S.~Klinkel.
\newblock {S}mooth multi-patch scaled boundary isogeometric analysis for
  {K}irchhoff–{L}ove shells.
\newblock {\em Meccanica}, 58(8):1693--1716, 2023.

\bibitem{RiAuFe16}
A.~Riffnaller-Schiefer, U.~H. Augsd\"orfer, and D.W. Fellner.
\newblock Isogeometric shell analysis with {NURBS} compatible subdivision
  surfaces.
\newblock {\em Applied Mathematics and Computation}, 272:139--147, 2016.

\bibitem{SaJu21}
A.~Sailer and B.~J\"uttler.
\newblock Approximately ${C}^1$-smooth isogeometric functions on two-patch
  domains.
\newblock In {\em Isogeometric Analysis and Applications 2018}, pages 157--175.
  Springer, LNCSE, 2021.

\bibitem{TaDe14}
A.~Tagliabue, L.~Ded\`{e}, and A.~Quarteroni.
\newblock Isogeometric analysis and error estimates for high order partial
  differential equations in fluid dynamics.
\newblock {\em Computers $\&$ Fluids}, 102:277 -- 303, 2014.

\bibitem{TaTo22}
T.~Takacs and D.~Toshniwal.
\newblock {Almost-${C}^1$ splines: Biquadratic splines on unstructured
  quadrilateral meshes and their application to fourth order problems}.
\newblock {\em Comp. Methods Appl. Mech. Engrg.}, 403:115640, 2023.

\bibitem{ToSpHiHu16}
D.~Toshniwal, H.~Speleers, R.~Hiemstra, and T.~J.~R. Hughes.
\newblock Multi-degree smooth polar splines: A framework for geometric modeling
  and isogeometric analysis.
\newblock {\em Comput. Methods Appl. Mech. Engrg.}, 316:1005--1061, 2017.

\bibitem{ToSpHiMaHu2020}
D.~Toshniwal, H.~Speleers, R.~R. Hiemstra, C.~Manni, and T.~J.~R. Hughes.
\newblock Multi-degree {B}-splines: Algorithmic computation and properties.
\newblock {\em Computer Aided Geometric Design}, 76:101792, 2020.

\bibitem{ToSpHu17}
D.~Toshniwal, H.~Speleers, and T.~J.~R. Hughes.
\newblock Smooth cubic spline spaces on unstructured quadrilateral meshes with
  particular emphasis on extraordinary points: Geometric design and
  isogeometric analysis considerations.
\newblock {\em Comput. Methods Appl. Mech. Engrg.}, 327:411--458, 2017.

\bibitem{VeWeMaTaTo24}
H.M. Verhelst, P.~Weinm\"uller, A.~Mantzaflaris, T.~Takacs, and D.~Toshniwal.
\newblock A comparison of smooth basis constructions for isogeometric analysis.
\newblock {\em Comput. Methods Appl. Mech. Engrg.}, 419:116659, 2024.

\bibitem{WangYe2013}
J.~Wang and X.~Ye.
\newblock A weak {G}alerkin finite element method for second-order elliptic
  problems.
\newblock {\em Journal of Computational and Applied Mathematics}, 241:103--115,
  2013.

\bibitem{WeLiQiHuZhCa22}
X.~Wei, X~Li, K.~Qian, T.~J.~R Hughes, Y.~J Zhang, and H.~Casquero.
\newblock Analysis-suitable unstructured {T}-splines: {M}ultiple extraordinary
  points per face.
\newblock {\em Comput. Methods Appl. Mech. Engrg.}, 391:114494, 2022.

\bibitem{WeTa21}
P.~Weinm\"uller and T.~Takacs.
\newblock {Construction of approximate ${C}^1$ bases for isogeometric analysis
  on two-patch domains}.
\newblock {\em Comput. Methods Appl. Mech. Engrg.}, 385:114017, 2021.

\bibitem{WeTa22}
P.~Weinmüller and T.~Takacs.
\newblock {An approximate ${C}^1$ multi-patch space for isogeometric analysis
  with a comparison to Nitsche’s method}.
\newblock {\em Comp. Methods Appl. Mech. Engrg.}, 401:115592, 2022.

\bibitem{WeFaLiWeCa23}
Wen. Z., Md.~S. Faruque, X.~Li, X.~Wei, and H.~Casquero.
\newblock Isogeometric analysis using {G}-spline surfaces with arbitrary
  unstructured quadrilateral layout.
\newblock {\em Comput. Methods Appl. Mech. Engrg.}, 408:115965, 2023.

\bibitem{ZhSaCi18}
Q.~Zhang, M.~Sabin, and F.~Cirak.
\newblock Subdivision surfaces with isogeometric analysis adapted refinement
  weights.
\newblock {\em Computer-Aided Design}, 102:104--114, 2018.

\end{thebibliography}

\end{document}